\documentclass[12pt]{amsart}

\author{Paul \textsc{Poncet}}
\address{CMAP, \'{E}cole Polytechnique, Route de Saclay, 91128 Palaiseau Cedex, France} 

\email{poncet@cmap.polytechnique.fr}

\usepackage[colorinlistoftodos,bordercolor=orange,backgroundcolor=orange!20,linecolor=orange,textsize=scriptsize]{todonotes}
\usepackage{stmaryrd}
\usepackage{amssymb}
\usepackage{amsmath}
\usepackage{calrsfs} 

\usepackage{times} 

\newcommand{\x}{w_{x}} 
\newcommand{\y}{w_{y}} 
\newcommand{\z}{w_{z}} 
\newcommand{\xoy}{p_{x, y}} 
\newcommand{\yoz}{p_{y, z}} 
\newcommand{\xoz}{p_{x, z}} 

\newtheorem*{theorem*}{Theorem}

\newtheorem{theorem}{Theorem}[section]
\newtheorem{corollary}[theorem]{Corollary}
\newtheorem{proposition}[theorem]{Proposition}
\newtheorem{lemma}[theorem]{Lemma}

\theoremstyle{definition}
\newtheorem{definition}[theorem]{Definition}
\newtheorem{example}[theorem]{Example}

\newtheorem{remark}[theorem]{Remark}

\begin{document}

\title{Partial metrics and \\ normed inverse semigroups}

\date{\today}



\begin{abstract}
Relying on the notions of submodular function and partial metric, we introduce normed inverse semigroups as a generalization of normed groups and sup-semilattices equipped with an upper valuation. 
We define the property of skew-convexity for a metric on an inverse semigroup, and prove that every norm on a Clifford semigroup gives rise to a right-subinvariant and skew-convex metric; it makes the semigroup into a Hausdorff topological inverse semigroup if the norm is cyclically permutable. 
Conversely, we show that every Clifford monoid equipped with a right-subinvariant and skew-convex metric admits a norm for which the metric topology and the norm topology coincide.  
We characterize convergence of nets and show that Cauchy completeness implies conditional monotone completeness with respect to the natural partial order of the inverse semigroup. 
\end{abstract}

\maketitle

\section{Introduction}

Informally, a norm on a set $S$ is a real-valued map on $S$ meant to bring quantitative information on top of the qualitative information already conveyed by the structure preexisting on $S$. 
In addition, one seeks to define a metric from the norm and make some topological information available as well. 

In line with the classical notion of normed vector spaces, norms have been exported to groups: normed groups have proved useful in the study of topological groups, see e.g.\ Bingham and Ostaszewski \cite{Bingham10}, Sarfraz et al.\ \cite{Sarfraz20}. 
In this case, there is an obvious correspondence between norms and \textit{right-invariant} metrics, i.e.\ metrics $d$ such that
\[
d(x - y, z - y) = d(x, z),
\]
for all $x, y, z$. 

On sup-semilattices, equivalently commutative idempotent semigroups, the difference between two elements does not exist; so, given a real-valued map $v$ playing the role of a norm, one cannot write something like $v(x - y)$ to derive a metric. 
To circumvent this hurdle, one can consider an \textit{upper valuation} on a sup-semilattice $(S, \vee)$, defined as 
a map $v : S \to \mathbb{R}_+$ satisfying the \textit{submodularity} condition, i.e.\ such that 
\[
v(x \vee y) + v(z) \leqslant v(x \vee z) + v(z \vee y),
\]
for all $x, y, z \in S$; see e.g.\ Ramana Murty and Engelbert \cite{Ramana85}, Schellekens \cite{Schellekens04}, Simovici \cite{Simovici14}. 
If $S$ is a lattice, this latter condition holds if and only if $v$ is order-preserving (aka non-decreasing) and 
\[
v(x \vee y) + v(x \wedge y) \leqslant v(x) + v(y),
\]
for all $x, y \in S$. 
Given an upper valuation on a sup-semilattice, it is easy to derive a pseudo-metric $d$ as 
\[
d(x, y) = { v(x \vee y) - \frac{v(x) + v(y)}{2} }, 
\]
which becomes a metric if $v$ is increasing. 
Moreover, $d$ is \textit{right-subinvariant} (actually, \textit{subinvariant} since $S$ is commutative), in the sense that
\[
d(x \vee y, z \vee y) \leqslant d(x, z),
\]
for all $x, y, z$, and \textit{radially-convex}, i.e.\ 
\[
d(x, z) = d(x, y) + d(y, z),
\]
for all $x \leqslant y \leqslant z$. 
Submodular maps play an important role in various mathematical areas such as combinatorial optimization (see Fujishige \cite{Fujishige05}) or information theory, where the Shannon entropy, seen as a map defined on the lattice of measurable finite partitions of a probability space, is submodular (see Nakamura \cite[p.~454]{Nakamura70}). 
See also Barth\'{e}l\'{e}my \cite{Barthelemy78}, Monjardet \cite{Monjardet81}, Orum and Joslyn \cite{Orum09} for upper valuations defined on a poset instead of a sup-semilattice. 

Since groups and sup-semilattices are both special kinds of semigroups, it is natural to look for a notion of norm that covers both group norms and sup-semilattice upper valuations. 
Recently, extensions of norms to semigroups have been considered, see e.g.\ Krishnachandran \cite{Krishnachandran15}, Dikranjan and Giordano Bruno \cite{Dikranjan19}. 
However, these authors mimic the axioms of a group norm, which notably comprise subadditivity, without including any axiom of submodularity flavor. 
Thus, their definition does not encompass upper valuations, and is of limited practical usage for sup-semilattices; especially, deriving a metric from a semigroup norm in Krishnachandran's sense is not obvious. 


In this paper, our goal is to fill this gap. 
Instead of tackling semigroups in their full generality, we focus on inverse semigroups, in the continuation of previous work undertaken by the author in \cite{Poncet11} and \cite{Poncet12a}. 
Recall that an \textit{inverse semigroup} is a semigroup $(S, +)$ such that, for all $x \in S$, there exists a unique $x^* \in S$ with $x + x^* + x = x$ and $x^* + x + x^* = x^*$. 
To reach our target and give an adequate notion of norm on inverse semigroups, we rely on the concept of partial metric and partial metric space introduced by Matthews \cite{Matthews92}. 
A partial metric $p$ is similar to a metric, except that self-distances $p(x, x)$ are not necessarily zero, yet are kept `small' with the condition $p(x, x) \leqslant p(x, y)$, for all $x, y$; also, subadditivity is replaced by a condition of submodularity that generalizes submodularity of maps on sup-semilattices. 
Indeed, given a set $S$, a map $p : S \times S \to \mathbb{R}$ is \textit{submodular} on $S$ if
\[
p(x, y) + p(z, z) \leqslant p(x, z) + p(z, y),
\]
for all $x, y, z \in S$. 
With this definition, a group norm $v$ induces a submodular map $(x, y) \mapsto v(x - y)$ (if the group is denoted additively), and a sup-semilattice upper valuation $v$ induces a submodular map $(x, y) \mapsto v(x \vee y)$. 
After Matthews, partial metric spaces have been further studied as a generalization of metric spaces, notably for extending various fixed point theorems, see e.g.\ Valero \cite{Valero05}, Samet et al.\ \cite{Samet13}, Han et al.\ \cite{Han17}, Bugajewski and Wang \cite{Bugajewski20}. 
They have also been applied to domain theory, the branch of mathematics originally developed for specifying denotational semantics in theoretical computer science; see the monograph by Gierz et al.\ \cite{Gierz03}, and see e.g.\ Waszkiewicz \cite{Waszkiewicz03} on correspondences between partial metrics and Martin's measurements on a continuous poset. 


The main result of this paper relies on the notion of skew-convex metric, which generalizes the property of radial-convexity and will be properly defined in Section~\ref{sec:clifford}. 
In a slightly simplified version, it goes as follows:

\begin{theorem*}
Let $(S, +, 0)$ be an inverse monoid such that $x + x^* = x^* + x$, for all $x \in S$ (i.e., $S$ is a Clifford monoid). 
If $\| \cdot \|$ is a norm on $S$, then the map $d$ defined on $S \times S$ by 
\[
d(x, y) = { \| x + y^* \| - \frac{\| x + x^* \| + \| y + y^* \| }{2} }, 
\]
for all $x, y \in S$, is a right-subinvariant and skew-convex metric on $S$. 
Moreover, if $\| \cdot \|$ is cyclically-permutable, in the sense that $\| x + y \| = \| y + x \|$, for all $x, y \in S$, then $S$ is a Hausdorff topological inverse semigroup with respect to the topology induced by $d$ (called the norm topology), and a monotone net $(x_n)_n$ in $S$ converges if and only if $\{ \| x_n \| \}_n$ is bounded. 
In addition, if $S$ is Cauchy complete for $d$, then $S$ is conditionally monotone-complete, and the natural partial order of the inverse semigroup $S$ is compatible with the norm topology. 
Conversely, if $d$ is a right-subinvariant, skew-convex metric on $S$, then the map $\| \cdot \|$ defined on $S$ by 
\[
\| x \| := d(x, 0) + d(x + x^*, 0),
\]
for all $x \in S$, is a norm on $S$ such that the metric topology induced by $d$ and the norm topology coincide. 
\end{theorem*}

Our motivation comes from the search for a unifying framework of both classical and idempotent analysis, as explained in \cite{Poncet11}. 
Idempotent analysis is a well established theory dating back to Zimmermann \cite{Zimmermann77} and popularized by Maslov \cite{Maslov87}; see also Kolokoltsov and Maslov \cite{Kolokoltsov89a},  \cite{Kolokoltsov89b}, Litvinov \cite{Litvinov11}. 
For such a program to be meaningful, one needs to dispose of tools such as a unified notion of Bochner integral, which would integrate measurable $M$-valued functions, for some ``unified space'' $M$ encompassing both topological vector spaces and topological idempotent semimodules. 
We expect such a space to be equipped with an adequate notion of norm making it into a Cauchy complete topological space.

The paper is organized as follows. 
Section~\ref{sec:posets} is a reminder of some usual definitions and notations from poset theory. 
In Section~\ref{sec:sub}, we recall the notion of submodular map and give some first properties. 
Inspired by Topkins \cite{Topkis78}, we show how to construct new submodular maps from a given one. 
In Section~\ref{sec:interlaced}, we introduce \textit{interlaced spaces} as a generalization of Matthews' notion of partial metric space. 
We show how an interlaced space can be made into a metric space and equipped with a partial order; we characterize convergence of nets, and provide various other properties that lay the ground for Section~\ref{sec:is}. 
Section~\ref{sec:dist2} focuses on a specific additional metric that can be built on a partial metric space; the triangle inequality property, which is not obvious, is proved with care. 
In Section~\ref{sec:is}, we introduce normed inverse semigroups, apply the results of the previous sections, and prove the first part of our main result. 
In Section~\ref{sec:clifford}, we define the notion of skew-convex metric, show that skew-convexity implies radial-convexity, and prove that, in the special case of Clifford monoids, every right-subinvariant, skew-convex metric gives rise to a norm for which the metric topology and the norm topology coincide.

\section{Reminders on posets and notations}\label{sec:posets}

\subsection{Qosets and posets}

A \textit{quasiordered set} or \textit{qoset} $(P,\leqslant)$ is a set $P$ together with a reflexive and transitive binary relation $\leqslant$. 
If in addition $\leqslant$ is antisymmetric, then $(P, \leqslant)$ is a \textit{partially ordered set} or \textit{poset}. 
Let $P$ be a qoset and $A \subseteq P$. 
An \textit{upper bound} of $A$ is an element $x \in P$ such that $a \leqslant x$ for all $a \in A$. 
\textit{Lower bounds} are defined dually. 
We write $A^{\uparrow}$ (resp.\ $A^{\downarrow}$) for the set of upper bounds (resp.\ lower bounds) of $A$. 
The subset $A$ is \textit{bounded} if both $A^{\uparrow}$ and $A^{\downarrow}$ are nonempty. 
If $x \in P$, then $\{ x \}^{\uparrow}$ (resp.\ $\{ x \}^{\downarrow}$) is the \textit{principal filter} (resp.\ the \textit{principal ideal}) generated by $x$. 
A \textit{supremum} (or \textit{sup}) of $A$ is any element of $A^{\uparrow} \cap (A^{\uparrow})^{\downarrow}$; if the quasiorder is a partial order, then a sup, if it exists, is necessarily unique, and then denoted by $\bigvee A$. 
An \textit{infimum} (or \textit{inf}) of $A$ is defined dually. 
A nonempty subset $D$ of $P$ is \textit{directed} if, for all $d, d' \in D$, there exists some $d'' \in D$ such that $d \leqslant d''$ and $d' \leqslant d''$. 
A map $f : P \to P'$ between qosets $P, P'$ is \textit{order-preserving} if $f(x) \leqslant f(y)$ whenever $x \leqslant y$. 
A \textit{sup-semilattice} (resp.\ a \textit{lattice}) is a poset in which every pair $\{ x, y \}$ has a sup (resp.\ a sup and an inf). 
The sup (resp.\ the inf) of $\{ x, y \}$ is denoted by $x \vee y$ (resp.\ by $x \wedge y$) if it exists. 

\subsection{Other notations}

We write $\mathbb{R}$ (resp.\ $\mathbb{R}_+$) for the set of real numbers (resp.\ nonnegative real numbers).

\section{Submodularity}\label{sec:sub}

\subsection{Submodular maps}

Given a set $X$, let $p : X \times X \to { \mathbb{R} \cup \{ -\infty \} }$. 
The map $p$ is said to be \textit{submodular} if
\[
p(x, y) + p(z, z) \leqslant p(x, z) + p(z, y),
\]
for all $x, y, z \in X$, and \textit{symmetric} if $p(x, y) = p(y, x)$, for all $x, y \in X$. 
We customary denote by $w_p$ (or $w$ if the context is clear) the map $x \mapsto p(x, x)$. 
We say that $p$ (resp.\ $w_p$) is \textit{finite} if it never takes the value $-\infty$.

\begin{remark}\label{rk:bart}
Following Barth\'{e}l\'{e}my \cite{Barthelemy78}, we define an \textit{upper valuation} on a directed qoset $P$ as an order-preserving map $w : P \to { \mathbb{R} \cup \{ -\infty \} }$ such that $p$ is submodular, where $p$ is the symmetric map $P \times P \to { \mathbb{R} \cup \{ -\infty \} }$ defined by 
\begin{equation}\label{eq:wplus}
p(x, y) = \inf \{ w(z) : x \leqslant z, y \leqslant z \}, 
\end{equation}
for all $x, y \in P$. 
See also Monjardet \cite{Monjardet81}, Orum and Joslyn \cite{Orum09}. 
\end{remark}

\begin{example}
Let $f, g : X \to { \mathbb{R} \cup \{ -\infty \} }$. 
Then $(x, y) \mapsto f(x)$, $(x, y) \mapsto g(y)$, and $(x, y) \mapsto f(x) + g(y)$ are submodular, and $(x, y) \mapsto f(x) \vee f(y)$ is symmetric submodular.
\end{example}

\begin{example}\label{ex:topkins}
Let $n$ be a positive integer and $\alpha$ be a real number $\geqslant 1$. 
Then the map $\mathbb{R}_+^n \times \mathbb{R}_+^n \to \mathbb{R}_+, (x, y) \mapsto \sum_{i=1}^n (x_i \vee y_i)^{\alpha}$ is finite and submodular, see Topkis \cite[Theorem~3.3]{Topkis78}. 
\end{example}

\begin{lemma}\label{lem:firstprops2}
If $p$ is a submodular map on $X$, then
\[
w_p(x) + w_p(y) \leqslant p(x, y) + p(y, x),
\]
for all $x, y \in X$, so that $p$ is finite if and only if $w_p$ is finite. 
If moreover $p$ is symmetric, then 
\[
w_p(x) + w_p(y) \leqslant 2 p(x, y),
\]
for all $x, y \in X$. 
\end{lemma}

\begin{proof}
Simply apply the submodularity condition above with $x = y$. 
\end{proof}

\begin{lemma}\label{lem:firstprops1}
If $p$ is a finite submodular map on $X$, then it is symmetric if and only if 
\[
p(x, y) + p(z, z) \leqslant p(x, z) + p(y, z),
\]
for all $x, y, z \in X$.  
\end{lemma}

\begin{proof}
The `only if' part is clear. 
Now suppose that $p$ satisfies the inequality of the lemma; then, with $z = x$, we obtain $p(x, y) + p(x, x) \leqslant p(x, x) + p(y, x)$, hence $p(x, y) \leqslant p(y, x)$, and similarly $p(y, x) \leqslant p(x, y)$, so that $p(x, y) = p(y, x)$. 
\end{proof}

To any map $p : X \times X \to \mathbb{R} \cup \{ -\infty \}$ we associate the quasiorder $\leq_p$ on $X$ defined by $x \leq_p y$ if $p(x, x) \leqslant p(y, y)$ and $p(x, z) \leqslant p(y, z)$, for all $z \in X$. 

\begin{lemma}\label{lem:leq}
Let $p$ be a finite, submodular map on $X$. 
Consider $X$ as a qoset endowed with the quasiorder $\leq_p$. 
Then the maps $w_p : x \mapsto p(x, x)$ and $x \mapsto p(x, z)$ ($z \in X$) are order-preserving.  
Moreover, $x \leq_p y$ if and only if $w_p(x) \vee p(x, y) \leqslant w_p(y)$, and $\leq_p$ is a partial order if and only if 
\begin{equation}\label{eq:partial}
{ w_p(x) = p(x, y) = p(y, x) = w_p(y) } \Rightarrow { x = y },
\end{equation}
for all $x, y \in X$.
\end{lemma}

\begin{proof}
It is clear from the definitions that the maps $w_p : x \mapsto p(x, x)$ and $x \mapsto p(x, z)$ ($z \in X$) are order-preserving. 

Let $x, y \in X$. 
If $x \leq_p y$, then $w_p(x) \leqslant w_p(y)$ and $p(x, y) \leqslant p(y, y) = w_p(y)$, so that 
\begin{equation}\label{eq:partial2}
w_p(x) \vee p(x, y) \leqslant w_p(y).
\end{equation}
Conversely, suppose that \eqref{eq:partial2} holds. 
Then $w_p(x) \leqslant w_p(y)$ on the one hand. 
On the other hand, if $z \in X$, then $p(x, z) + p(x, y) \leqslant p(x, z) + w_p(y) \leqslant p(x, y) + p(y, z)$, hence $p(x, z) \leqslant p(y, z)$. 
So $x \leq_p y$. 

Now assume that Condition~\eqref{eq:partial} is satisfied, 
and suppose that $x \leq_p y$ and $y \leq_p x$. 
Then $w_p(x) \leqslant w_p(y)$, $p(x, y) \leqslant w_p(y)$, $w_p(y) \leqslant w_p(x)$, $p(y, x) \leqslant w_p(x)$. 
Moreover, $x \leq_p y$ also implies $w_p(x) \leqslant p(y, x)$, and $y \leq_p x$ implies $w_p(y) \leqslant p(x, y)$. 
All in all, we obtain $w_p(x) = w_p(y) = p(x, y) = p(y, x)$. 
So $x = y$. 

Conversely, assume that $\leq_p$ is a partial order, and suppose that $w_p(x) = p(x, y) = p(y, x) = w_p(y)$. 
Then $x \leq_p y$ and $y \leq_p x$, hence $x = y$. 
\end{proof}

\begin{remark}[Remark~\ref{rk:bart} continued]\label{rk:stable}
Let $(P, \leqslant)$ be a qoset and $w : P \to { \mathbb{R} \cup \{ -\infty \} }$ be an upper valuation on $P$. 
It is easily proved that the quasiorder $\leqslant$  is contained in $\leq_p$ and can be replaced by $\leq_{p}$ in Equation~\eqref{eq:wplus} of Remark~\ref{rk:bart}, i.e.\ 
\begin{equation}\label{eq:stable}
p(x, y) = \inf \{ w(z) : x \leq_{p} z, y \leq_{p} z \},
\end{equation}
for all $x, y \in P$. 
Following Waszkiewicz \cite[Definition~3.3]{Waszkiewicz01}, we say that a symmetric submodular map $p$ is \textit{stable} if Equation~\eqref{eq:stable} holds with $w := w_p$. 
Then we have a one-to-one correspondence between (finite) upper valuations and (finite) stable symmetric submodular maps. 
Moreover, a stable symmetric submodular map $p$ satisfies 
$
w(x) \leqslant p(x, y)
$, 
for all $x, y$; this implies that $p$ is necessarily a \textit{partial pseudo-metric} (see the definition in Section~\ref{subsec:ppm}) if $w$ takes only non-negative real values. 
\end{remark}

\subsection{Constructing new submodular maps}

A first way to construct new submodular maps is the following. 

\begin{proposition}
If $p$ is a submodular map on $X$, 
then the map 
\[
q : (x, y) \mapsto { p(x, x) \vee p(x, y) }
\]
is submodular. 
Moreover, $w_q$ agrees with $w_p$, and $\leq_q$ agrees with $\leq_p$. 
\end{proposition}

\begin{proof}
The proof is straightforward. 
\end{proof}

\begin{proposition}
If $p$ is a submodular map on $X$, 
then the maps $(x, y) \mapsto f(x) + p(x, y)$ and $(x, y) \mapsto p(x, y) + f(y)$ are submodular, 
for every map $f : X \to \mathbb{R}$. 
In particular, if $p$ is finite, then $(x, y) \mapsto p(x, y) - p(x, x)$ is submodular. 
\end{proposition}

\begin{proof}
The proof is straightforward. 
\end{proof}

Here is a more involved result, inspired by Topkis \cite[Table~1]{Topkis78}. 

\begin{proposition}\label{prop:fgsub}
Let $p$ be a finite submodular map on $X$ such that 
\[
p(x, x) \vee p(y, y) \leqslant p(x, y),
\]
for all $x, y \in X$, and let $f : K \to { \mathbb{R} \cup \{ -\infty \} }$ be a concave order-preserving map on a convex subset $K$ of $\mathbb{R}$ with ${ K } \supseteq { p(X \times X) }$. 
Then $f \circ p$ is submodular. 
\end{proposition}

\begin{proof}
We write $w$ for the map $x \mapsto p(x, x)$. 
Let $x, y, z \in X$. 
We want to show the following inequality: 
\begin{equation}\label{eq:subfg}
f \circ p(x, y) + f \circ w(z) \leqslant f \circ p(x, z) + f \circ p(z, y). 
\end{equation}
If one of $f \circ p(x, y), f \circ w(z), f \circ p(x, z), f \circ p(z, y)$ is equal to $-\infty$, then \eqref{eq:subfg} holds. 
So now assume that all of these terms are finite. 
Take
\begin{align*}
F_1 &:= f \circ p(x, z) - f(P), \\
F_2 &:= f(P) - f \circ p(x, y) - f \circ w(z) + f \circ p(z, y), \\
P &:=  p(x, y) \vee p(z, y) + w(z) - p(z, y). 
\end{align*}
Then \eqref{eq:subfg} is the same as $F_1 + F_2 \geqslant 0$. 
Note that $P \geqslant w(z)$ so $f(P)$ is finite. 
Using the properties of $p$ and the fact that $f$ is order-preserving, we deduce that $F_1 \geqslant 0$. 
So let us show that $F_2 \geqslant 0$ to conclude the proof. 

\textit{Case}~1: $p(x, y) \leqslant p(z, y)$. 
Then $P = w(z)$ and 
\[
F_2 = f \circ p(z, y) - f \circ p(x, y) \geqslant 0. 
\]

\textit{Case}~2: $w(z) = p(z, y) < p(x, y)$. 
Then $P = p(x, y)$ and $F_2 = 0$. 

\textit{Case}~3: $w(z) < p(z, y) < p(x, y)$. 
Then the concavity of $f$ implies 
\[
\frac{f(p(x, y)) - f(w(z))}{p(x, y) - w(z)} \geqslant \frac{f(p(x, y)) - f(p(z, y))}{p(x, y) - p(z, y)}. 
\]
Moreover, $w(z) < P < p(x, y)$, so the concavity of $f$ also implies
\[
\frac{f(P) - f(w(z))}{P - w(z)} \geqslant \frac{f(p(x, y)) - f(w(z))}{p(x, y) - w(z)}. 
\]
Hence we have 
\[
\frac{f(P) - f(w(z))}{P - w(z)} \geqslant \frac{f(p(x, y)) - f(p(z, y))}{p(x, y) - p(z, y)}. 
\]
And since $P - w(z) = p(x, y) - p(z, y) > 0$ this shows that 
\[
f(P) - f \circ w(z) \geqslant f \circ p(x, y) - f \circ p(z, y), 
\]
i.e.\ $F_2 \geqslant 0$, and the result is proved. 
\end{proof}

\begin{example}
Let $p$ be a finite submodular map such that 
\[
0 \leqslant p(x, x) \vee p(y, y) \leqslant p(x, y),
\]
for all $x, y$. 
Then $(x, y) \mapsto -1/p(x, y)$, $(x, y) \mapsto \log p(x, y)$ and $(x, y) \mapsto \sqrt{ p(x, y) }$ are submodular, so that 
\begin{align*}
p(x, z)^{-1} + p(z, y)^{-1} &\leqslant p(x, y)^{-1} + p(z, z)^{-1}, \\
p(x, y) p(z, z) &\leqslant p(x, z) p(z, y), \\
\sqrt{ p(x, y) } + \sqrt{ p(z, z) } &\leqslant \sqrt{ p(x, z) } + \sqrt{ p(z, y) }, 
\end{align*}
for all $x, y, z$. 
\end{example}

\begin{example}\label{ex:transfo3}
Let $b > 0$ and $p$ be a finite submodular map such that
\[
-b \leqslant p(x, x) \vee p(y, y) \leqslant p(x, y),
\]
for all $x, y$. 
Then $(x, y) \mapsto \frac{p(x, y)}{b + p(x, y)}$ is submodular.
See K\"unzi and Vajner \cite{Kuenzi94}.
\end{example}

\section{Partial metrics and interlaced spaces}\label{sec:interlaced}

In this section, we introduce \textit{(pseudo-)interlaced spaces} as a generalization of Matthews' notion of partial (pseudo-)metric space, see Matthews \cite{Matthews92}, Waszkiewicz \cite{Waszkiewicz01}, Bukatin et al.\ \cite{Bukatin09}. 

\subsection{Reminders on partial (pseudo-)metrics}\label{subsec:ppm}

A \textit{partial pseudo-metric} on a set $X$ is a finite symmetric submodular map $p : X \times X \to \mathbb{R}$ such that
\[
0 \leqslant w_p(x) \leqslant p(x, y),
\]
for all $x, y \in X$. 
A pair $(X, p)$ is then called a \textit{partial pseudo-metric space}. 
A \textit{partial metric} on $X$ is a partial pseudo-metric $p$ on $X$ such that 
\[
{ w_p(x) = p(x, y) = w_p(y) } \Rightarrow { x = y }, 
\]
for all $x, y \in X$. 
In this case, $(X, p)$ is called a \textit{partial metric space}. 

\subsection{Pseudo-interlaced spaces}

We now generalize the concept of partial pseudo-metric space by introducing the following notion. 

\begin{definition}
A \textit{pseudo-interlaced space} is a triplet $(X, p, q)$, where $X$ is a set, $p$ and $q$ are maps $X \times X \to \mathbb{R}$, and the following conditions are satisfied:
\begin{itemize}
  \item $p$ is finite symmetric submodular on $X$;
  \item ${-q}$ is finite symmetric submodular on $X$;
  \item $w(x) := p(x, x) = q(x, x)$, for all $x \in X$;
  \item there is some $k > 0$ such that 
  $
  w(x) + k q(x, y) \leqslant k p(x, y) + w(y)
  $,
  for all $x, y \in X$. 
\end{itemize}
\end{definition}

\begin{example}
Let $(X, d)$ be a \textit{pseudo-metric space}, in the sense that the map $d : X \times X \to \mathbb{R}_+$ satisfies $d(x, x) = 0$ and $d(x, y) = d(y, x)$, for all $x, y \in X$, and the triangle inequality holds:
\[
d(x, z) \leqslant d(x, y) + d(y, z),
\]
for all $x, y, z \in X$. 
Then $d$ is called a \textit{pseudo-metric} on $X$, and $(X, d, 0)$ is a pseudo-interlaced space, where $0$ denotes the real-valued map $(x, y) \mapsto 0$. 
\end{example}

\begin{lemma}\label{lem:ineq}
Let $(X, p, q)$ be a pseudo-interlaced space. 
Then
\[
q(x, y) \leqslant \frac{w(x) + w(y)}{2} \leqslant p(x, y),
\]
for all $x, y \in X$. 
\end{lemma}

\begin{proof}
This is a direct consequence of Lemma~\ref{lem:firstprops2} applied to $p$ and $-q$. 
\end{proof}

\begin{proposition}\label{prop:pi}
Let $(X, p, q)$ be a pseudo-interlaced space. 
Then the map $d_{p, q} : X \times X \to \mathbb{R}$ defined by 
\[
d_{p, q}(x, y) = p(x, y) - q(x, y)
\]
is a pseudo-metric on $X$ that makes $X$ a completely regular topological space, and $p$, $q$, and $w$ continuous maps.
Moreover, if $(x_n)_n$ is a net in $X$ and $x \in X$, then $x_n \to x$ with respect to $d_{p, q}$ if and only if $p(x_n, x) \to w(x)$ and $q(x_n, x) \to w(x)$. 
\end{proposition}

We call the topology induced by $d_{p, q}$ the \textit{$d_{p, q}$-topology}, and $d_{p, q}$ the \textit{intrinsic pseudo-metric} on $(X, p, q)$. 

\begin{proof}
As per the previous lemma we have $d_{p, q}(x, y) \geqslant 0$, for all $x, y \in X$. 
Since $p$ and $q$ are symmetric, we also have $d_{p, q}(x, y) = d_{p, q}(y, x)$, for all $x, y \in X$. 
The triangle inequality is obtained by summing up the submodularity inequalities of $p$ and ${-q}$, and using the fact that $p(z, z) = q(z, z)$ for all $z \in X$.
So $d_{p, q}$ is indeed a pseudo-metric on $X$. 
The fact that $X$ is a completely regular topological space under the $d_{p, q}$-topology is a classical result. 


Let us show that $w$ is continuous. 
There exists some $k > 0$ such that $w(x) + k q(x, y) \leqslant k p(x, y) + w(y)$, for all $x, y \in X$. 
Using the symmetry of $p$ and $q$, this implies that 
\begin{equation}\label{eq:lip}
| w(x) - w(y) | \leqslant k d_{p, q}(x, y), 
\end{equation}
for all $x, y \in X$. 
Thus, $w$ is Lipschitz-continuous, hence continuous. 

Now we show that $p$ and $q$ are continuous. 
Let $((x_n, y_n))_n$ be a net such that $x_n \to x$ and $y_n \to y$. 
By Lemma~\ref{lem:ineq}, we have 
\[
0 \leqslant p(x_n, x) - \frac{w(x_n) + w(x)}{2} \leqslant d_{p, q}(x_n, x),
\]
so that $p(x_n, x) \to w(x)$. 
Similarly, $p(y_n, y) \to w(y)$. 
Applying the submodularity property of $p$ multiple times, we obtain
\[
p(x_n, y_n) - p(x, y) \geqslant w(x_n) - p(x_n, x) + w(y_n) - p(y_n, y),
\]
and 
\[
p(x_n, y_n) - p(x, y) \leqslant p(x_n, x) - w(x) + p(y_n, y) - w(y),
\]
so that $p(x_n, y_n) \to p(x, y)$. 
This proves that $p$ is continuous. 
The continuity of $q$ can be proved along similar lines. 

To complete the proof, let $(z_n)_n$ be a net and $z \in X$. 
Suppose that $p(z_n, z) \to w(z)$ and $q(z_n, z) \to w(z)$. 
Then $d(z_n, z) = p(z_n, z) - q(z_n, z)$ tends to $w(z) - w(z) = 0$, i.e.\ $z_n$ tends to $z$. 
\end{proof}

\begin{remark}\label{rk:eqpq}
Suppose that $p(x, y) = q(x, y)$, for some $x, y \in X$. 
Then, using Inequality~\eqref{eq:lip} in the proof of Proposition~\ref{prop:pi}, we get $w(x) = w(y)$, hence $w(x) = \frac{w(x) + w(y)}{2} = w(y)$. 
From Lemma~\ref{lem:ineq}, we obtain $w(x) = p(x, y) = q(x, y) = w(y)$. 
\end{remark}

\begin{example}\label{ex:q0q1}
Let $(X, p)$ be a partial pseudo-metric space. 
We say that a map $q : X \times X \to \mathbb{R}$ is \textit{adjoint to $(X, p)$} if ${-q}$ is a finite symmetric submodular map such that $w(x) := q(x, x) = p(x, x)$, for all $x \in X$. 
In this case, $(X, p, q)$ is a pseudo-interlaced space. 
To see why this holds, it suffices to observe that $w(x) + 2 q(x, y) \leqslant 2 p(x, y) + w(y)$, for all $x, y \in X$ (use Lemma~\ref{lem:ineq} and the fact that $w(x) \leqslant p(x, y)$, by definition of a partial pseudo-metric). 

Note that the maps 
\begin{align*}
q_0 &: (x, y) \mapsto w(x) \wedge w(y), \\
q_1 &: (x, y) \mapsto \frac{w(x) + w(y)}{2}, 
\end{align*}
are adjoint to $(X, p)$. 
This notably provides us with the following pseudo-metrics on $X$: 
\[
d_{p, 0} : (x, y) \mapsto p(x, y) - w(x) \wedge w(y)
\]
and 
\[
d_{p, 1} : (x, y) \mapsto p(x, y) - \frac{w(x) + w(y)}{2}, 
\]
that we call respectively the \textit{zeroth and first intrinsic pseudo-metrics} on $(X, p)$. 
Both $d_{p, 0}$ and $d_{p, 1}$ are \textit{radially-convex}, in the sense that 
\[
d(x, z) = d(x, y) + d(y, z),
\]
whenever $x \leq_p y \leq_p z$ and $d \in \{ d_{p, 0}, d_{p, 1} \}$. 
Indeed, let $x, y, z \in X$ with $x \leq_p y \leq_p z$. 
Then $p(x, y) = w(y)$, $p(y, z) = w(z)$, and $p(x, z) = w(z)$. 
Moreover, $d_{p, 1}(x, z) = p(x, z) - \frac{w(x) + w(z)}{2} = \frac{w(z) - w(x)}{2}$, and similarly $d_{p, 1}(x, y) = \frac{w(y) - w(x)}{2}$ and $d_{p, 1}(y, z) = \frac{w(z) - w(y)}{2}$, so that $d_{p, 1}(x, z) = d_{p, 1}(x, y) + d_{p, 1}(y, z)$. 
This shows that $d_{p, 1}$ is radially-convex, and the proof for $d_{p, 0}$ is analogous. 
\end{example}

\begin{example}
Let $L$ be a lattice and let $f : L \to \mathbb{R}$ be a \textit{valuation} on $L$, i.e.\ an order-preserving map satisfying the \textit{modularity condition}
\[
f(x \vee y) + f(x \wedge y) = f(x) + f(y),
\]
for all $x, y \in L$. 
Then $(L, p)$ is a partial pseudo-metric space with adjoint $q$, where $p$ and $q$ are the maps $p : (x, y) \mapsto f(x \vee y)$ and $q : (x, y) \mapsto f(x \wedge y)$, and the corresponding pseudo-metric $d := d_{p, q}$ is defined by 
\[
d(x, y) = f(x \vee y) - f(x \wedge y),
\]
for all $x, y \in L$. 
Moreover, $d$ is a metric if and only if $f$ is increasing, i.e.\ $x < y$ implies $f(x) < f(y)$, for all $x, y \in L$. 
\end{example}

\begin{proposition}\label{prop:eq}
Let $(X, p, q)$ be a pseudo-interlaced space. 
Assume that 
\[
w(x) + w(y) \leqslant p(x, y) + q(x, y),
\]
for all $x, y \in X$. 
Then $d_{p, q}$ is equivalent to $d_{p, 1}$, and we have
\[
d_{p, 1} \leqslant d_{p, q} \leqslant 2 d_{p, 1}, 
\]
so that the $d_{p, q}$-topology coincides with the $d_{p, 1}$-topology. 
Moreover, if $(x_n)_n$ is a net and $x \in X$, then $x_n \to x$ with respect to $d_{p, q}$ if and only if $p(x, x_n) \to w(x)$ and $w(x_n) \to w(x)$. 
\end{proposition}

\begin{proof}
Using Lemma~\ref{lem:ineq}, the inequality $d_{p, 1} \leqslant d_{p, q}$ is clear. 
Now let $x, y \in X$. 
By hypothesis, $w(x) + w(y) \leqslant p(x, y) + q(x, y)$, hence $d_{p, q}(x, y) = p(x, y) - q(x, y) \leqslant 2 p(x, y) - w(x) - w(y) = 2 d_{p, 1}(x, y)$. 
This proves that $d_{p, q}$ is equivalent to $d_{p, 1}$. 

Let $x \in X$ and $(x_n)_n$ be a net. 
If $x_n \to x$ with respect to $d_{p, q}$, then $p(x, x_n) \to w(x)$ by Proposition~\ref{prop:pi}, and $w(x_n) \to w(x)$ by continuity of $w$. 
Conversely, if $p(x, x_n) \to w(x)$ and $w(x_n) \to w(x)$, then $d_{p, 1}(x, x_n) \to 0$, hence $d_{p, q}(x, x_n) \to 0$, as required. 
\end{proof}

A pseudo-interlaced space $(X, p, q)$ can be equipped with a quasiorder $\leq_{p, q}$ defined by $x \leq_{p, q} y$ if $x \leq_{p} y$ and $y \leq_{-q} x$. 
Using Lemma~\ref{lem:leq}, we have $x \leq_{p, q} y$ if and only if $w(x) \leqslant q(x, y) \leqslant p(x, y) \leqslant w(y)$, for all $x, y \in X$. 

\begin{proposition}\label{prop:closedorder}
Let $(X, p, q)$ be a pseudo-interlaced space equipped with its $d_{p, q}$-topology. 
Then the quasiorder $\leq_{p, q}$ is a closed subset of $X \times X$, and principal ideals and principal filters are closed subsets of $X$. 
\end{proposition}

\begin{proof}
Let $((x_n, y_n))_n$ be a net in $X \times X$ with $x_n \to x$, $y_n \to y$, and $x_n \leq_{p, q} y_n$ for all $n$. 
From $x_n \leq_{p, q} y_n$ we get $w(x_n) \leqslant q(x_n, y_n) \leqslant p(x_n, y_n) \leqslant w(y_n)$. 
Now, the continuity of $p$, $q$, and $w$ yields $w(x) \leqslant q(x, y) \leqslant p(x, y) \leqslant w(y)$, which shows that $x \leq_{p, q} y$.
This proves that $\leq_{p, q}$ is a closed subset of $X \times X$.  
The fact that principal ideals and principal filters with respect to $\leq_{p, q}$ are closed is a straightforward consequence. 
\end{proof}

\subsection{Interlaced spaces}

An \textit{interlaced space} is a pseudo-interlaced space $(X, p, q)$ such that ${ p(x, y) = q(x, y) } \Rightarrow { x = y }$, for all $x, y \in X$.  

\begin{proposition}\label{prop:tych}
Let $(X, p, q)$ be a pseudo-interlaced space. 
Then the following conditions are equivalent: 
\begin{enumerate}
  \item\label{prop:tych5} $(X, p, q)$ is interlaced;
  \item\label{prop:tych1} $(X, d_{p, q})$ is Tychonoff;
  \item\label{prop:tych2} $(X, d_{p, q})$ is Hausdorff; 
  \item\label{prop:tych3} $(X, d_{p, q})$ is $T_0$;
  \item\label{prop:tych4} $d_{p, q}$ is a metric;
  \item\label{prop:tych6} $\leq_{p, q}$ is a partial order.
\end{enumerate}
\end{proposition}

\begin{proof}
The equivalences \eqref{prop:tych1} $\Leftrightarrow$ \eqref{prop:tych2} $\Leftrightarrow$ \eqref{prop:tych3} $\Leftrightarrow$ \eqref{prop:tych4} are classical, and \eqref{prop:tych4} $\Leftrightarrow$ \eqref{prop:tych5} is obvious.  

\eqref{prop:tych5} $\Leftrightarrow$ \eqref{prop:tych6}. 
Recall that $x \leq_{p, q} y$ if and only if $w(x) \leqslant q(x, y) \leqslant p(x, y) \leqslant w(y)$. 
Then, using Remark~\ref{rk:eqpq}, $\leq_{p, q}$ is a partial order if and only if ${ p(x, y) = q(x, y) } \Rightarrow { x = y }$, for all $x, y \in X$, if and only if $(X, p, q)$ is interlaced. 
\end{proof}

\begin{proposition}
Let $(X, p, q)$ be a pseudo-interlaced space. 
Then the binary relation $\sim$ defined on $X$ by $x \sim y$ if $d_{p, q}(x, y) = 0$ is an equivalence relation that makes the quotient set $X /\!\! \sim$ into an interlaced space whose intrinsic metric topology coincides with the quotient topology. 
\end{proposition}

\begin{proof}
It is obvious that $\sim$ is an equivalence relation on $X$. 
We write $\tilde{X}$ for the quotient set $X /\!\! \sim$, and $\pi : X \to \tilde{X}$ for the quotient map. 
If $x \in X$, we write $\tilde{x}$ as a shorthand for $\pi(x)$, i.e.\ $\tilde{x}$ is the equivalence class of $x$ in $\tilde{X}$. 
It is not difficult to show that $x \sim x'$ and $y \sim y'$ imply $p(x, y) = p(x', y')$ and $q(x, y) = q(x', y')$, for all $x, y, x', y' \in X$. 
Thus, it is valid to define the maps $\tilde{p}$ and $\tilde{q}$ on $\tilde{X} \times \tilde{X}$ by $\tilde{p}(\tilde{x}, \tilde{y}) = p(x, y)$ (resp.\ $\tilde{q}(\tilde{x}, \tilde{y}) = q(x, y)$), for all $x, y \in X$. 
Then $(\tilde{X}, \tilde{p}, \tilde{q})$ is an interlaced space.  

To show that the $d_{\tilde{p}, \tilde{q}}$-topology coincides with the quotient topology on $\tilde{X}$, we show that $\pi^{-1}(F')$ is closed in $X$ if and only if $F'$ is closed in the $d_{\tilde{p}, \tilde{q}}$-topology, for all subsets $F'$ of $\tilde{X}$. 
Suppose first that $\pi^{-1}(F')$ is closed in $X$, and let $(\tilde{x}_n)_n$ be a net in $F'$ that converges to some $\tilde{x}$ in the $d_{\tilde{p}, \tilde{q}}$-topology. 
Then $d_{p, q}(x_n, x) = p(x_n, x) - q(x_n, x) = \tilde{p}(\tilde{x}_n, \tilde{x}) - \tilde{q}(\tilde{x}_n, \tilde{x}) = d_{\tilde{p}, \tilde{q}}(\tilde{x}_n, \tilde{x}) \to 0$. 
Since $\pi^{-1}(F')$ is closed in $X$, this implies that $x \in \pi^{-1}(F')$, hence that $\tilde{x} = \pi(x) \in F'$. 
So $F'$ is closed in the $d_{\tilde{p}, \tilde{q}}$-topology. 
Now suppose that $F'$ is closed in the $d_{\tilde{p}, \tilde{q}}$-topology, and let $(x_n)_n$ be a net in $\pi^{-1}(F')$ that converges to some $x$ in $X$. 
Then $d_{\tilde{p}, \tilde{q}}(\tilde{x}_n, \tilde{x}) = d_{p, q}(x_n, x) \to 0$. 
Since $F'$ is closed in the $d_{\tilde{p}, \tilde{q}}$-topology, this implies that $\tilde{x} \in F'$, hence that $x \in \pi^{-1}(F')$. 
So $\pi^{-1}(F')$ is closed in $X$. 
\end{proof}

Hereunder, we call \textit{monotone} a net that is either non-decreasing or non-increasing. 
The \textit{supinf} of a monotone net $(x_n)_n$, if it exists, is the sup (resp.\ the inf) of $\{ x_n \}_n$ if the net is non-decreasing (resp.\ non-increasing). 
A poset is \textit{(conditionally) monotone-complete} if every (bounded) monotone net has a supinf. 
On a set, a topology and a partial order are \textit{compatible} if, whenever an element $x$ is the supinf of a monotone net, then the net converges to $x$. 

\begin{theorem}\label{thm:cvi}
Let $(X, p, q)$ be an interlaced space. 
Assume that $(X, d_{p, q})$ is Cauchy complete. 
Then $(X, \leq_{p, q})$ is conditionally monotone-complete, and the partial order $\leq_{p, q}$ is compatible with the topology. 
Moreover, if $(x_n)_n$ is a monotone net, then the following conditions are equivalent:
\begin{enumerate}
  \item\label{thm:cvi1} $(x_n)_n$ converges;
  \item\label{thm:cvi1bis} $(x_n)_n$ has a supinf;
  \item\label{thm:cvi2} $\{ x_n \}_n$ is bounded;
  \item\label{thm:cvi3} $\{ w(x_n) \}_n$ is bounded.
\end{enumerate}
\end{theorem}

\begin{proof}
Let $(x_n)_n$ be a monotone net. 

\eqref{thm:cvi1} $\Rightarrow$ \eqref{thm:cvi1bis}. 
Suppose that $x_n \to x$. 
Using the monotonicity of $(x_n)_n$ and the fact that principal ideals and principal filters are closed by Proposition~\ref{prop:closedorder}, we easily deduce that $x$ is the supinf of $(x_n)_n$. 

\eqref{thm:cvi1bis} $\Rightarrow$ \eqref{thm:cvi2} is straightforward. 

\eqref{thm:cvi2} $\Rightarrow$ \eqref{thm:cvi3}. 
If $\{ x_n \}_n$ is bounded, there are some $\ell, u \in X$ with $\ell \leq_{p, q} x_n \leq_{p, q} u$ for all $n$. 
Since $w$ is order-preserving, we have $w(\ell) \leqslant w(x_n) \leqslant w(u)$ for all $n$, so $\{ w(x_n) \}_n$ is bounded. 

\eqref{thm:cvi3} $\Rightarrow$ \eqref{thm:cvi1}. 
Suppose that $\{ w(x_n) \}_n$ is bounded in $\mathbb{R}$. 
We show that $(x_n)_n$ is a Cauchy net. 
If $m \leqslant n$, then $x_m \leq_{p, q} x_n$ or $x_n \leq_{p, q} x_m$, from which we deduce that $d_{p, q}(x_m, x_n) \leqslant | w(x_n) - w(x_m) |$. 
The real net $(w(x_n))_n$ is monotone and bounded, hence is a Cauchy net. 
So $(x_n)_n$ is a Cauchy net; since $(X, d_{p, q})$ is Cauchy complete by hypothesis, we deduce that $(x_n)_n$ converges to some $x$. 

Let us show that the $d_{p, q}$-topology and the partial order $\leq_{p, q}$ are compatible. 
If $x$ is the supinf of a monotone net $(x_n)_n$, then the previous point shows that $(x_n)_n$ converges to some $y$. 
By the first part of the proof, $y$ is necessarily the supinf of $(x_n)_n$, i.e.\ $y = x$. 

The fact that $(X, \leq_{p, q})$ is conditionally monotone-complete follows from the implication \eqref{thm:cvi2} $\Rightarrow$ \eqref{thm:cvi1bis}. 
\end{proof}

\section{The second intrinsic pseudo-metric}\label{sec:dist2}

In this section, given a partial pseudo-metric space $(X, p)$, we introduce a pseudo-metric $d_{p, 2}$ that generates the same topology as the pseudo-metrics $d_{p, 0}$ and $d_{p, 1}$ of Example~\ref{ex:q0q1}, without being equivalent to them. 
The proof of the triangle inequality is not straightforward, so we need the following technical lemma first. 

\begin{lemma}\label{lem:dist2}
Let $(X, p)$ be a partial pseudo-metric space. 
Then, for all $x, y, z \in X$, we have 
\begin{equation}\label{eq:lemdist2}
2 \Gamma + \Theta \leqslant 2 \sqrt{\Delta},  
\end{equation}
where 
\begin{align*}
\Gamma &:= (p(x, z) - w_p(z))(p(y, z) - w_p(z)), \\
\Theta &:= (w_p(z) - w_p(x))(w_p(y) - w_p(z)), \\
\Delta &:= (p(x, z)^2 - w_p(x) w_p(z))(p(y, z)^2 - w_p(y) w_p(z)). 
\end{align*}
\end{lemma}

\begin{proof}
For ease of notation, we write $p_{x, y}$ for $p(x, y)$, $w_{x}$ for $w_p(x)$, etc. 
We distinguish between the cases $\Theta \leqslant 0$ and $\Theta > 0$. 

\textit{Case}~1: $\Theta \leqslant 0$. 
Using $\x + \z \leqslant 2 \xoz$, we obtain $(\xoz - \z)^2 \leqslant \xoz^2 - \x \z$. 
Similarly, $(\yoz - \z)^2 \leqslant \yoz^2 - \y \z$. 
This implies that $\Gamma^2 \leqslant \Delta$.
Since $\Theta \leqslant 0$ we get Equation~\eqref{eq:lemdist2}. 

\textit{Case}~2: $\Theta > 0$. 
We can suppose without loss of generality that $\x \leqslant \y$, so that $\x \leqslant \z \leqslant \y$. 
Elementary calculus gives 
\begin{align*}
\Gamma^2 = &\left( \xoz^2 - \x \z + \z ( \x + \z - 2 \xoz) \right) \\
  &\times \left( \yoz^2 - \y \z + \z ( \y + \z - 2 \yoz) \right) \\
= {}& \Delta - \z \left( \xoz^2 - \x \z \right) \left( 2 \yoz - \y - \z \right) \\
  &- \z \left( \yoz^2 - \y \z \right) \left( 2 \xoz - \x - \z \right) \\
  &+ \z^2 \left( 2 \xoz - \x - \z \right) \left( 2 \yoz - \y - \z \right). 
\end{align*}
We get 
\begin{align*}
\Gamma^2 - \Delta \leqslant &- \z^2 \left( \xoz - \x \right) \left( 2 \yoz - \y - \z \right) \\
  &- \y \z \left( \yoz - \z \right) \left( 2 \xoz - \x - \z \right) \\ 
  &+ \z^2 \left( 2 \xoz - \x - \z \right) \left( 2 \yoz - \y - \z \right) \\
= &- \y \z \left( \yoz - \z \right) \left( 2 \xoz - \x - \z \right) \\
  &+ \z^2 \left( \xoz - \z \right) \left( 2 \yoz - \y - \z \right) \\
= &- \y \z \Gamma - \y \z \left( \yoz - \z \right) \left( \xoz - \x \right) \\
  &+ \z^2 \Gamma + \z^2 \left( \xoz - \z \right) \left( \yoz - \y \right). \\
\end{align*}
Consider now the term $T = \z^2 \left( \xoz - \z \right) \left( \yoz - \y \right)$ on the right hand side of the previous inequality. 
Since $\x \leqslant \z \leqslant \y$ we have 
\begin{align*}
T &\leqslant \z^2 \left( \xoz - \x \right) \left( \yoz - \y \right) \\ 
&\leqslant \y \z \left( \xoz - \x \right) \left( \yoz - \y \right) \\ 
&= \y \z \left( \xoz - \x \right) \left( \yoz - \z \right) - \y \z \left( \xoz - \x \right) \left( \y - \z \right). 
\end{align*}
Hence, 
\[
{ \Gamma^2 - \Delta } \leqslant { - \y \z \Gamma + \z^2 \Gamma - \y \z \left( \xoz - \x \right) \left( \y - \z \right) }.
\]
For $\Theta$ we give the following (rough) upper bound: 
\begin{align*}
\Theta &= -\z^2 - \x \y + \x \z + \y \z \\
&\leqslant \x \z + \y \z \leqslant 2 \y \z \leqslant 4 \y \z,
\end{align*}
so that 
\[
{ \frac{1}{4} \Theta^2 } \leqslant { \y \z \Theta } \leqslant { \y \z (\xoz - x) (\y - \z) }. 
\]
This implies that 
\begin{align*}
\Gamma^2 - \Delta + \frac{1}{4} \Theta^2 &\leqslant - \y \z \Gamma + \z^2 \Gamma 
\end{align*}
so that 
\[
{ \Gamma^2 - \Delta + \frac{1}{4} \Theta^2 + \Theta \Gamma } \leqslant { - \x (\y - \z )\Gamma } \leqslant { 0 }. 
\]
This proves that $(\Gamma + \Theta/2)^2 = \Gamma^2 + \frac{1}{4} \Theta^2 + \Theta \Gamma \leqslant \Delta$, i.e.\ that $2\Gamma \leqslant -\Theta + 2 \sqrt{\Delta}$, which is the desired result. 
\end{proof}

\begin{theorem}\label{thm:dist2}
Let $(X, p)$ be a partial pseudo-metric space. 
Then the map $d_{p, 2} : X \times X \to \mathbb{R}_+$ defined by 
\[
d_{p, 2}(x, y) = \sqrt{p(x, y)^2 - w_p(x) w_p(y) }, 
\]
is a pseudo-metric on $X$. 
\end{theorem}

We call $d_{p, 2}$ the \textit{second} \textit{intrinsic pseudo-metric} on $(X, p)$. 

\begin{proof}
For ease of notation, we write $p_{x, y}$ for $p(x, y)$, $w_{x}$ for $w_p(x)$, etc. 
Let $x, y, z \in X$, and let us show that $d_{p, 2}(x, y) \leqslant d_{p, 2}(x, z) + d_{p, 2}(z, y)$. 
With the submodularity of $p$ we have 
\[
\xoy \leqslant \xoz + \yoz - \z, 
\]
so that
\[
\xoy^2 \leqslant \xoz^2 + \yoz^2 + 2 \Gamma - \z^2, 
\]
where $\Gamma := (\xoz - \z)(\yoz - \z)$. 
Using Lemma~\ref{lem:dist2} this implies 
\[
\xoy^2 + \z^2 \leqslant \xoz^2 + \yoz^2 + (\x - \z)(\y - \z) + 2 \sqrt{\Delta},  
\]
where $\Delta := (\xoz^2 - \x \z)(\yoz^2 - \y \z)$. 
This rewrites as  
\[
\xoy^2 - \x \y \leqslant \xoz^2 - \x \z + \yoz^2 - \y \z + 2 \sqrt{\Delta}, 
\]
i.e.\ 
\[
d_{p, 2}(x, y)^2 \leqslant d_{p, 2}(x, z)^2 + d_{p, 2}(z, y)^2 + 2 d_{p, 2}(x, z) d_2(z, y), 
\]
which yields $d_{p, 2}(x, y) \leqslant d_{p, 2}(x, z) + d_{p, 2}(z, y)$. 
Moreover, $d_{p, 2}(x, y) = d_{p, 2}(y, x)$ and $d_{p, 2}(x, x) = 0$, for all $x, y \in X$. 
This shows that $d_{p, 2}$ is a pseudo-metric. 
\end{proof}

\begin{proposition}
Let $(X, p)$ be a partial pseudo-metric space. 
Then 
\[
d_{p, 0} \leqslant 2 d_{p, 1} \leqslant 2 (d_{p, 0} \wedge d_{p, 2}). 
\]
In particular, $d_{p, 0}$ are $d_{p, 1}$ are equivalent. 
\end{proposition}

\begin{proof}
To lighten the notations, we write $p_{x, y}$ for $p(x, y)$, $w_{x}$ for $w_p(x)$, etc. 
The inequalities $d_{p, 0} \leqslant 2 d_{p, 1} \leqslant 2 d_{p, 0}$ are easily shown. 
Now let us prove that $d_{p, 1} \leqslant d_{p, 2}$. 
Let $x, y \in X$. 
It can be seen that 
\[
p_{x, y} (w_{x} + w_{y}) \geqslant w_{x}^2 + w_{y}^2. 
\]
This implies 
\begin{align*}
4 p_{x, y} (w_{x} + w_{y}) &\geqslant 4 w_{x}^2 + 4 w_{y}^2 \\
&\geqslant w_{x}^2 + w_{y}^2  + 6 w_{x} w_{y} + 3(w_{y} - w_{x})^2 \\
&\geqslant w_{x}^2 + w_{y}^2  + 6 w_{x} w_{y}. 
\end{align*}
The latter inequality is equivalent to 
\[
4 (p_{x, y}^2 - w_{x} w_{y}) \geqslant (2 p_{x, y} - w_{x} - w_{y})^2, 
\]
so that 
$
d_{p, 2}(x, y) \geqslant d_{p, 1}(x, y)
$. 
\end{proof}

We delay to the next section the data of a counterexample showing that $d_{p, 2}$ is not equivalent to $d_{p, 0}$ and $d_{p, 1}$ in general (see Example~\ref{ex:counter}), even though all these pseudo-metrics give rise to the same topology, as asserted by the following result. 

\begin{proposition}\label{prop:lim2}
Let $(X, p)$ be a partial pseudo-metric space.  
Then the intrinsic pseudo-metrics $d_{p, 0}$, $d_{p, 1}$, and $d_{p, 2}$ generate the same topology. 
\end{proposition}

\begin{proof}
We already know that the pseudo-metrics $d_{p, 0}$ and $d_{p, 1}$ generate the same topology. 
Also, the inequality 
\[
w_p(x)^2 + w_p(y)^2 \leqslant p(x, y)^2 + w_p(x) w_p(y)
\]
holds for all $x, y \in X$. 
(To prove it, it suffices to distinguish between the cases $w_p(x) \leqslant w_p(y)$ and $w_p(x) > w_p(y)$.) 
As a consequence, 
\[
w_p(x)^2 + w_p(y)^2 - 2 w_p(x) w_p(y) \leqslant d_{p, 2}(x, y)^2 = p(x, y)^2 - w_p(x) w_p(y), 
\]
so that
\begin{equation}\label{eq:contnorm2}
| w_p(x) - w_p(y) | \leqslant d_{p, 2}(x, y),  
\end{equation}
for all $x, y \in X$. 
This proves that $w$ is continuous with respect to the pseudo-metric $d_{p, 2}$. 

To prove that the topologies generated by $d_{p, 0}$ and $d_{p, 2}$ coincide, it suffices to show by Proposition~\ref{prop:eq} that a net $(x_n)$ in $X$ converges to $x$ with respect to $d_{p, 2}$ if and only if $w_p(x_n) \to w_p(x)$ and $p(x_n, x) \to w_p(x)$. 
First assume that $x_n$ tends to $x$ with respect to $d_{p, 2}$, i.e.\ that $d_{p, 2}(x_n, x) \to 0$. 
By \eqref{eq:contnorm2}, $w_p(x_n) \to w_p(x)$, so $p(x_n, x) = \sqrt{ d_{p, 2}(x_n, x)^2 + w_p(x_n) w_p(x) } \to w_p(x)$. 
Conversely, assume that $w_p(x_n) \to w_p(x)$ and $p(x_n, x) \to w_p(x)$. 
Then $d_{p, 2}(x_n, x) = \sqrt{ p(x_n, x)^2 - w_p(x_n) w_p(x) }$ tends to $0$, i.e.\ $x_n$ tends to $x$ with respect to $d_{p, 2}$. 
\end{proof}

\section{Pseudo-normed and normed inverse semigroups}\label{sec:is}

\subsection{Reminders on inverse semigroups}

A \textit{semigroup} $(S, +)$ (denoted additively) is a set $S$ equipped with an associative binary relation $+$ (the addition). 
Be aware that we do not suppose the addition to be commutative in general, despite the usage of additive notation. 
An element $e$ of $S$ is \textit{idempotent} if $e + e = e$. 
A semigroup $S$ is \textit{inverse} if the idempotent elements of $S$ commute and if, for all $x \in S$, there is some $y \in S$, called an \textit{inverse} of $x$, such that $x + y + x = x$ and $y + x + y = y$. 
A semigroup is inverse if and only if every element $x$ has a unique inverse, denoted by $x^*$. 
We write $E(S)$ for the commutative subsemigroup of an inverse semigroup $S$ made of its idempotent elements; this is a sup-semilattice.  
An \textit{inverse monoid} is an inverse semigroup with an \textit{identity element}, i.e.\ an element $0$ such that $0 + x = x + 0 = x$, for all $x$. 
On inverse semigroup theory, we refer the reader to the monograph by M.\ V.\ Lawson \cite{Lawson98b}.  
Let us recall some basic facts.

\begin{lemma}
Let $S$ be an inverse semigroup. 
Then $S$ can be equipped with a partial order $\leqslant_S$, compatible with the semigroup structure, and defined by $x \leqslant_S y$ if $y = x + e$ for some $e \in E(S)$. 
Moreover, if $x, y \in S$, then:
\begin{itemize}
  \item $x + x^*$ and $x^* + x$ are idempotent;
  \item $(x^*)^* = x$;
  \item $(x + y)^* = y^* + x^*$;
  \item $x^* = x$ if $x$ is idempotent;
  \item $(x \leqslant_S y) \Leftrightarrow (x^* \leqslant_S y^*) \Leftrightarrow (y = x + y^* + y) \Leftrightarrow (y = y + y^* + x)$. 
\end{itemize}
\end{lemma}

\begin{proof}
See \cite[Chapter~1]{Lawson98b}. 
\end{proof}

We also recall the following classical example of inverse semigroup, see \cite[Section~3.4]{Lawson98b}. 

\begin{example}\label{ex:bm}
Let $(G, +, 0)$ be a lattice-ordered group and $G_+$ be its non-negative part, i.e.\ $G_+ := \{ x \in G : x \geqslant 0 \}$.  
On $G_+ \times G_+$, one can define the binary relation $+$ by 
\[
(a, b) + (c, d) = (a - b + b \vee c, d - c + b \vee c).
\]
This makes $G_+ \times G_+$ into an inverse monoid with identity $(0, 0)$, called the
\textit{bicyclic monoid} on $G$. 
It satisfies $(a, b)^* = (b, a)$, and an element $(a, b)$ is idempotent if and only if $a = b$. 
Moreover, the intrinsic order on $G_+ \times G_+$ satisfies $(a, b) \leqslant (c, d)$ if and only if $0 \leqslant c - a = d - b$. 
In particular, for idempotent elements, $(a, a) \leqslant (b, b)$ if and only if $a \leqslant b$. 
Note also that the injection $i_G : G \to G_+ \times G_+$, $x \mapsto (x \vee 0, (-x) \vee 0)$ satisfies
\begin{itemize}
  \item $i_G(0) = (0, 0)$, and
  \item $i_G(-x) = i_G(x)^*$, and
  \item $i_G(x + y) \leqslant i_G(x) + i_G(y)$ (subadditivity),
\end{itemize}
for all $x, y \in G$. 
\end{example}

On an inverse semigroup $S$, we shall denote by $\delta$ the map $S \to S$, $x \mapsto { x + x^* }$, and often write $\delta x$ instead of $\delta(x)$, so that $x = \delta x + x$ and $x^* + \delta x = x^*$, for all $x \in S$. 

\begin{lemma}\label{lem:clifford}
Let $S$ be an inverse semigroup. 
Then the following assertions are equivalent:
\begin{enumerate}
  \item\label{lem:clifford1} $x + x^* = x^* + x$, for all $x \in S$;
  \item\label{lem:clifford2} $e + x = x + e$, for all $e \in E(S)$, $x \in S$;
  \item\label{lem:clifford3} $x + \delta x = x$, for all $x \in S$;
  \item\label{lem:clifford4} $\delta(x + y) = \delta(x) + \delta(y)$, for all $x, y \in S$. 
\end{enumerate}
\end{lemma}

\begin{proof}
\eqref{lem:clifford1} $\Rightarrow$ \eqref{lem:clifford2}. 
Let $e \in E(S)$, $x \in S$. 
Then $e + x = (e + x + x^*) + x = (x + x^* + e) + x = x + (e + x)^* + (e + x) = x + (e + x) + (e + x)^* = x + (e + x + x^* + e) = x + (x + x^* + e) = (x + x^* + x) + e = x + e$. 

\eqref{lem:clifford2} $\Rightarrow$ \eqref{lem:clifford4}. 
Let $x, y \in S$. 
Then $\delta(x + y) = x + \delta y + x^* = \delta y + x + x^* = \delta y + \delta x = \delta x + \delta y$, as required. 

\eqref{lem:clifford4} $\Rightarrow$ \eqref{lem:clifford3}. 
Let $x \in S$. 
Then $x + \delta(x) = x + \delta(x^*) + \delta(x) = x + \delta(x^* + x) = x + x^* + x = x$. 

\eqref{lem:clifford3} $\Rightarrow$ \eqref{lem:clifford1}. 
Let $x \in S$. 
Then $x = x + \delta x$, hence $x^* + x = (x^* + x) + (x + x^*)$. 
Thus, $x + x^* \leqslant_S x^* + x$. 
Analogously, $x^* + x \leqslant_S x + x^*$, so that $x + x^* = x^* + x$. 
\end{proof}

An inverse semigroup satisfying the conditions of the previous lemma is called a \textit{Clifford semigroup}. 
We deduce that an inverse semigroup $S$ is Clifford if and only if the map $\delta$ is an endomorphism of $S$. 
Note that groups and sup-semilattices are always Clifford semigroups. 

\subsection{Pseudo-norm on an inverse semigroup}

Let us start applying to inverse semigroups the concepts and results developed in the previous sections. 

\begin{definition}\label{def:pseudonorm}
A \textit{pseudo-normed inverse semigroup} is an inverse semigroup $(S, +)$ equipped with a \textit{pseudo-norm}, that is a map $S \to \mathbb{R}_+, x \mapsto \| x \|$ such that $(x, y) \mapsto \| x + y^* \|$ is a partial pseudo-metric, in the sense that the following properties are satisfied:
\begin{itemize}
  \item $\| x + y^* \| = \| y + x^* \|$, for all $x, y \in S$;
  \item $\| x + x^* \| \leqslant \| x + y^* \|$, for all $x, y \in S$; 
  \item $\| x + y^* \| + \| z + z^* \| \leqslant \| x + z^* \| + \| z + y^* \|$, for all $x, y, z \in S$. 
\end{itemize}
A \textit{pseudo-normed inverse monoid} is an inverse monoid $(S, +, 0)$ endowed with a map $S \to \mathbb{R}_+, x \mapsto \| x \|$ such that $(S, +, \| \cdot \|)$ is a pseudo-normed inverse semigroup and $\| 0 \| = 0$. 
\end{definition}

\begin{remark}\label{rk:pstar}
On a pseudo-normed inverse semigroup, the map $p^*$ defined as $(x, y) \mapsto \| x^* + y \|$ is also a partial pseudo-metric. 
We have $p^*(x, y) = p(x^*, y^*)$, for all $x, y$. 
However, $p$ and $p^*$ do not coincide in general. 
\end{remark}

\begin{proposition}\label{prop:normproperties}
Let $S$ be a pseudo-normed inverse semigroup, and let $p$ be the partial pseudo-metric $(x, y) \mapsto \| x + y^* \|$. 
If $x, y \in S$, then  
\begin{enumerate}
  \item\label{prop:normproperties0} $x \leqslant_S y \Rightarrow x \leq_p y$;
  \item\label{prop:normproperties2} $x \leq_p y \Rightarrow \| x \| \leqslant \| y \|$; 
  \item\label{prop:normproperties1} $\| x^* \| = \| x \|$; 
  \item\label{prop:normproperties3} $\| \delta x \| \leqslant \| x \|$; 
  \item\label{prop:normproperties4} $\| \delta x \| \vee \| \delta y \| \leqslant \| x + y^* \|$; 
  \item\label{prop:normproperties5} $\| \delta x \| + \| \delta y \| \leqslant 2 \| x + y^* \|$;
  \item\label{prop:normproperties6} $t(\delta x, \delta y) \leqslant t(x, y)$, 
\end{enumerate}
where $t : S \times S \to \mathbb{R}$ is defined by $t(x, y) = \| x^* \| + \| y \| - \| x^* + y \|$. 
\end{proposition}

\begin{proof}
\eqref{prop:normproperties0}.  
If $x \leqslant_S y$, then $y + y^* = x + y^*$, so that $p(y, y) = \| y + y^* \| = \| x + y^* \| = p(x, y)$. 
This shows that $x \leq_p y$. 

\eqref{prop:normproperties2}. 
At first, we show that $x \leqslant_S y \Rightarrow \| x \| \leqslant \| y \|$. 
So assume that $x \leqslant_S y$. 
Then $y = x + (y^* + y)$, and 
\begin{align*}
\| x \| + \| y^* + y \| &= \| x + (x^* + x) \| + \| (y^* + y) + (y^* + y)^* \| \\
&\leqslant \| x + (y^* + y)^* \| + \| (y^* + y) + (x^* + x) \| \\
&=  \| y \| + \| y^* + y \|, 
\end{align*}
hence $\| x \| \leqslant \| y \|$. 
Now assume that $x \leq_p y$. 
Then $p(x, y) = p(y, y)$, i.e.\ $\| x + y^* \| = \| y + y^* \|$. 
Note that $x \leqslant_S x + y^* + y$, so that $\| x \| \leqslant \| x + y^* + y \|$. 
Thus, $\| x \| + \| y + y^* \| \leqslant \| x + y^* + y \|  + \| y + y^* \| \leqslant \| x + y^* \| + \| y + y^* + y \| = \| y + y^* \| + \| y \|$, so that $\| x \| \leqslant \| y \|$. 

\eqref{prop:normproperties1}. 
Let $x \in S$, and take $y := x^* + x$. 
Then $x = x + y$ and $y = y^*$, which implies $\| x \| = \| x + y \| = \| x + y^* \| = \| y + x^* \| = \| x^* + x + x^* \| = \| x^* \|$. 

\eqref{prop:normproperties3}. 
Let $x \in S$, and take $y := x^* + x$. 
Then $x = x + y$ and $y = y^*$, which implies $\| \delta x \| = \| x + x^* \| \leqslant \| x + y^* \| = \| x + y \| = \| x \|$. 

\eqref{prop:normproperties4} is straightforward from the definitions. 

\eqref{prop:normproperties5} directly follows from \eqref{prop:normproperties4}. 

\eqref{prop:normproperties6}. 
Let $x, y \in S$. 
Then 
\begin{align*}
\| x^*  + y \| + \| \delta x \| + \| \delta y \| &\leqslant \| x^* + \delta x \| + \| \delta x + y \| + \| \delta y \| \\
&= \| x^* \| + \| \delta x + y \| + \| \delta y \| \\
&\leqslant \| x^* \| + \| \delta x + \delta y \| + \| \delta y + y \| \\
&= \| x^* \| + \| \delta x + \delta y \| + \| y \|, 
\end{align*}
so that $t(\delta x, \delta y) \leqslant t(x, y)$. 
\end{proof}

\begin{example}
If $S$ reduces to a group with identity element $0$ (with $x^* = {-x}$), then a pseudo-norm $\| \cdot \|$ on $S$ seen as a monoid corresponds to the concept of \textit{length}, in the sense that the following three conditions hold: 
\begin{itemize}
  \item $\| 0 \| = 0$; 
  \item $\| x \| = \| {-x} \|$, for all $x \in S$;
  \item $\| x + y \| \leqslant \| x \| + \| y \| $, for all $x, y \in S$. 
\end{itemize}
In this case, we say that $S$ is a \textit{pseudo-normed group}. 
\end{example}

\begin{example}
If $S$ reduces to a sup-semilattice (with $x^* = x$ and $x + y = x \vee y$), then a pseudo-norm on $S$ is exactly an \textit{upper valuation} (or a \textit{join semivaluation}, in Nakamura's terms \cite{Nakamura70}), 
i.e.\ a map $\| \cdot \| : S \to \mathbb{R}_+$ satisfying the submodularity condition: 
\[
\| x \vee y \| + \| z \| \leqslant \| x \vee z \| + \| z \vee y \|,
\]
for all $x, y, z \in S$. 
In this case, we call $S$ a \textit{pseudo-normed sup-semilattice}. 
If moreover $S$ is a lattice, i.e.\ if the inf $x \wedge y$ of $\{ x, y \}$ exists for all $x, y \in S$, the submodularity condition holds if and only if $\| \cdot \|$ is order-preserving and satisfies 
\[
\| x \vee y \| + \| x \wedge y \| \leqslant \| x \| + \| y \|,
\]
for all $x, y \in S$. 
See also Ramana Murty and Engelbert \cite{Ramana85}, Schellekens \cite{Schellekens04}, Simovici \cite{Simovici14}. 
\end{example}

\begin{example}[Example~\ref{ex:bm} continued]\label{ex:bm2}
Let $(G, +, 0)$ be a lattice-ordered group equipped with a group pseudo-norm $\| \cdot \|_G$. 
Then we have a pseudo-norm $\| \cdot \|$ on the bicyclic monoid $G_+ \times G_+$ defined by
\[
\| (a, b) \| := \| a - b \|_G,
\]
for all $a, b \in G_+$. 
Indeed, we have $\| (a, b)^* \| = \| (a, b) \|$ and $\| (a, a) \| = 0$, for all $a, b \in G_+$. 
Moreover, 
\begin{align*}
\| (a, b) + (c, d) \| &= \| (a - b + b \vee c, d - c + b \vee c) \| \\
&= \| (a - b) + (c - d) \|_G \\
&\leqslant \| a - b \|_G + \| c - d \|_G = \| (a, b) \| + \| (c, d) \|,
\end{align*}
for all $a, b, c, d \in G_+$, so $\| \cdot \|$ is subadditive. 
From there it is easy to conclude that $\| \cdot \|$ is a pseudo-norm. 
Note also that the injection $i_G$ is norm-preserving, i.e.\ $\| i_G(x) \| = \| x \|_G$, for all $x \in G$. 
\end{example}

\begin{remark}
If the pseudo-norm is \textit{homogeneous}, i.e.\ if $n \| x \| = \| n x \|$ for all $x \in S$ and all positive integers $n$, then $\| e \| = 0$ for every idempotent element $e \in S$.  
In this case, the submodularity condition implies subadditivity, i.e.\ $\| x + y \| \leqslant \| x \| + \| y \|$, for all $x, y \in S$. 
Indeed, with the notations of Proposition~\ref{prop:normproperties}, if $x, y \in S$, then we have $t(\delta(x^*), \delta(y)) = 0$, hence $0 \leqslant t(x^*, y)$, i.e.\ $\| x + y \| \leqslant \| x \| + \| y \|$. 
\end{remark}

\begin{remark}
Subadditivity also holds if $E(S)$ is filtered, in particular if $S$ is an inverse \textit{monoid}. 
Recall that a poset is \textit{filtered} if every two elements have a common lower bound. 
If $E(S)$ is filtered and $x, y \in S$, then $\delta(x^*)$ and $\delta(y)$ have a common lower bound $e \in E(S)$. 
This means  that $\delta(x^*) + e = \delta(x^*)$ and $e + \delta(y) = \delta(y)$. 
By submodularity, 
\[
\| \delta(x^*) + \delta(y) \| \leqslant \| \delta(x^*) + \delta(y) \| + \| e \| \leqslant \| \delta(x^*) \| + \| \delta(y) \|,
\] 
hence $0 \leqslant t(\delta(x^*), \delta(y)) \leqslant t(x^*, y)$. 
This shows that $\| x + y \| \leqslant \| x \| + \| y \|$. 
\end{remark}

\begin{remark}
If $S$ is a commutative inverse semigroup and $v$ is a pseudo-norm (an upper valuation) on the sup-semilattice $E(S)$, then $x \mapsto v(\delta x)$ is a pseudo-norm on $S$. 
\end{remark}


\begin{theorem}\label{thm:dist1}
Let $S$ be a pseudo-normed inverse semigroup. 
Then the maps $d_0, d_1, d_2$ defined on $S \times S$ by 
\begin{align*}
d_0(x, y) &= \| x + y^* \| - \| \delta x \| \wedge \| \delta y \|, \\
d_1(x, y) &= \| x + y^* \| - \frac{\| \delta x \| + \| \delta y \|}{2}, \\
d_2(x, y) &= \sqrt{\| x + y^* \|^2 - \| \delta x \| \| \delta y \|},
\end{align*}
for all $x, y \in S$, are pseudo-metrics on $S$ that generate the same completely regular topology. 
Moreover, if $(x_n)_{n}$ is a net in $S$ and $x \in S$, then $x_n \to x$ if and only if $\| x_n + x^* \| \to \| \delta x \|$ and $\| \delta x_n \| \to \| \delta x \|$. 
\end{theorem}

\begin{proof}
Combine Proposition~\ref{prop:lim2} and Proposition~\ref{prop:eq}. 
\end{proof}

We call the topology induced by any $d \in \{ d_0, d_1, d_2 \}$ the \textit{norm topology}. 

\begin{definition}\label{def:rad}
Let $S$ be an inverse semigroup and $d$ be a pseudo-metric on $S$. 
We say that $d$ is 
\begin{itemize}
  \item \textit{right-subinvariant} if 
$
d(x + y^*, z + y^*) \leqslant d(x, z)
$, 
for all $x, y, z \in S$; 
  \item \textit{radially-convex} if 
$
d(x, z) = d(x, y) + d(y, z) 
$, 
for all $x, y, z \in S$ with $x \leqslant_S y \leqslant_S z$. 
\end{itemize}
\end{definition}

The property of right-subinvariance will be at stake in Theorem~\ref{thm:exist} to characterize metrics giving rise to a norm on Clifford semigroups. 

If $e$ is an idempotent element in $S$, we write $S_e$ for the inverse subsemigroup of $S$ defined by
\[
S_e = \{ x \in S : x + e = e + x = x \}. 
\]
On $S_e$ we define the map $\| \cdot \|_e : S_e \to \mathbb{R}_+$ by
\[
\| x \|_e = \| x \| - \| e \|. 
\]
Then $(S_e, + , e, \| \cdot \|_e)$ is a pseudo-normed inverse monoid. 
Note that $\| \cdot \|_e$ indeed takes nonnegative values, since $\| x \| \geqslant \| \delta x \| \geqslant \| e \|$, for all $x \in S_e$. 

\begin{proposition}\label{prop:propd}
Let $S$ be a pseudo-normed inverse semigroup. 
Then the following assertions hold: 
\begin{enumerate}
  \item\label{prop:propd1} $d_0(x, e) = \| x \|_{e}$, for all $e \in E(S)$ and $x \in S_e$; 
  \item\label{prop:propd2} $d_0$ and $d_1$ are right-subinvariant; 
  \item\label{prop:propd3} $d_0$ and $d_1$ are radially-convex;
  \item\label{prop:propd7} $d(x, y) = d(\delta x, \delta y)$, for all $x, y \in S$ with a common lower bound and $d \in \{ d_0, d_1, d_2 \}$.
\end{enumerate}
\end{proposition}

\begin{proof}
\eqref{prop:propd1} is straightforward from the definitions. 

\eqref{prop:propd2}. 
Let $x, y, z \in S$. 
We apply the submodularity property twice. 
First, $\| x + (y + y^* + z^*) \| + \| z + z^* \| \leqslant \| x + z^* \| + \| \delta(z + y) \|$. 
And similarly, $\| (x + y + y^*) + z^*) \| + \| x + x^* \| \leqslant \| \delta(x + y) \| + \| x + z^* \|$. 
Taking the average of the two inequalities, we obtain 
\[
\| (x + y) + (z + y)^* \| + \frac{\| \delta x \| + \| \delta z \|}{2} \leqslant \| x + z^* \| + \frac{\| \delta(x + y) \| + \| \delta(z + y) \|}{2},
\] 
which is tantamount to saying that $d_1$ is right-subinvariant. 
If we take the infimum of the two inequalities instead, we obtain the same conclusion for $d_0$. 

\eqref{prop:propd3}. 
See Example~\ref{ex:q0q1}. 


\eqref{prop:propd7}. 
Let $x, y \in S$, and suppose that $z \leqslant_S x$ and $z \leqslant_S y$ for some $z \in S$. 
Then $x = x + x^* + z$ and $y^* = z^* + y + y^*$, hence $x + y^* = x + x^* + z + z^* + y + y^* = \delta x + \delta z + \delta y$. 
Moreover, $z \leqslant_S x$ implies $\delta z \leqslant_S \delta x$, so $x + y^* = \delta x + \delta y = \delta x + (\delta y)^*$. 
Now, it is easily seen that $d(x, y) = d(\delta x, \delta y)$, for $d \in \{ d_0, d_1, d_2 \}$. 
\end{proof}

We say that a pseudo-norm $\| \cdot \|$ on an inverse semigroup $S$ is \textit{cyclically permutable} if $\| x + y \| = \| y + x \|$, for all $x, y \in S$. 
This amounts to say that $p$ and $p^*$ agree, see Remark~\ref{rk:pstar}. 

\begin{corollary}\label{coro:topo}
Let $S$ be a pseudo-normed inverse semigroup endowed with the norm topology. 
Then the maps $x \mapsto x + y$ ($y \in S$) and $x \mapsto \| x \|$ are continuous. 
Moreover, if $\| \cdot \|$ is cyclically permutable, then the maps $x \mapsto x^*$, $x \mapsto \delta(x)$, and $(x, y) \mapsto x + y$ are continuous, and $S$ becomes a topological inverse semigroup. 
\end{corollary}

\begin{proof}
Let $y \in S$. 
The continuity of the map $x \mapsto x + y$ follows from the right-subinvariance of $d_1$ given by Proposition~\ref{prop:propd}\eqref{prop:propd2}. 

Let us show that the pseudo-norm $x \mapsto \| x \|$ is continuous. 
We denote by $p$ the partial pseudo-metric $(x, y) \mapsto \| x + y^* \|$, and by $w$ the map $x \mapsto p(x, x) = \| x  + x^* \|$. 
Let $(x_n)_n$ be a net with $x_n \to x$. 
Using the submodularity of $p$ and the fact that $x \leqslant_S x + x_n^* + x_n^{}$ for all $n$, we have
\[
\| x_n^{} \| \geqslant p(x, x_n^* + x_n^{}) + w(x_n^{}) - p(x, x_n^{}) \geqslant \| x \| + w(x_n^{}) - p(x, x_n^{}). 
\]
Moreover, from the fact that $x_n \leqslant_S x_n + x^* + x$ for all $n$, we also have
\[
\| x_n \| \leqslant \| x_n + x^* + x \| = p(x_n, x^* + x). 
\]
We know that $p$ and $w$ are continuous by Proposition~\ref{prop:pi}, hence 
\[
\| x \| + w(x_n) - p(x, x_n) \to \| x \| + w(x) - p(x, x) = \| x \|,
\]
and 
\[
p(x_n, x^* + x) \to p(x, x^* + x) = \| x \|.
\] 
This shows that $\| x_n \| \to \| x \|$, as required. 

Now assume that $\| \cdot \|$ is cyclically permutable. 
Then we have $d_1(x, y) = d_1(x^*, y^*)$, for all $x, y \in S$. 
Using the right-subinvariance property of Proposition~\ref{prop:propd}, we obtain 
\begin{align*}
d_1(x_1 + x_2, y_1 + y_2) &\leqslant d_1(x_1 + x_2, y_1 + x_2) + d_1(y_1 + x_2, y_1 + y_2) \\
&\leqslant d_1(x_1, y_1) + d_1(x_2^* + y_1^*, y_2^* + y_1^*) \\
&\leqslant d_1(x_1, y_1) + d_1(x_2^*, y_2^*) \\
&= d_1(x_1, y_1) + d_1(x_2, y_2),
\end{align*}
for all $x_1, x_2, y_1, y_2 \in S$. 
This implies the continuity of $(x, y) \mapsto x + y$. 
The continuity of the map $x \mapsto x^*$ is easily deduced from the cyclic permutability of $\| \cdot \|$. 
The continuity of $x \mapsto \delta(x)$ follows. 
\end{proof}

\begin{remark}
See also J.\ D.\ Lawson \cite{Lawson74} for sufficient conditions that make the addition of a semitopological semigroup continuous. 
\end{remark}

As announced at the end of the previous section, the following example shows that the pseudo-metric $d_{p, 2}$ is not equivalent to $d_{p, 0}$ and $d_{p, 1}$ in general. 

\begin{example}\label{ex:counter}
Let $S$ be a pseudo-normed sup-semilattice. 
Assume that the pseudo-norm is not identically zero and that there is a map 
\[
(\mathbb{R}_+ \setminus \{ 0 \}) \times S \to S, (\lambda, x) \mapsto \lambda \cdot x
\]
such that $\lambda \mapsto \lambda \cdot x$ is order-preserving, $1 \cdot x = x$, and $\| \lambda \cdot x \| = \lambda \| x \|$, for all $\lambda > 0$ and $x \in S$. 
Suppose that $d_{2} \leqslant \sqrt{k} d_{0}$, for some $k > 1$. 
Let $\lambda \in \mathbb{R}$ with $1 < \lambda < k/(k-1)$, $x \in S$ with $\| x \| \neq 0$, and $y = \lambda \cdot x$. 
Then $d_{2}(x, y)^2 = \lambda (\lambda - 1) \| x \|^2$ and $d_{0}(x, y)^2 = (\lambda -1)^2 \| x \|^2$. 
Thus, the inequality $d_{2}(x, y)^2 \leqslant k d_{0}(x, y)^2$ implies $\lambda \leqslant k (\lambda-1)$, a contradiction. 
So $d_{2}$ is not equivalent to $d_{0}$. 
\end{example}

\subsection{Normed inverse semigroups and completeness}

We now come to the notion of normed inverse semigroup announced in the Introduction and give our main result. 

\begin{definition}\label{def:norm}
A \textit{normed inverse semigroup} is a pseudo-normed inverse semigroup $(S, +, \| \cdot \|)$ such that the pseudo-norm $\| \cdot \|$ satisfies the \textit{separation condition} ${ \| x \|_e = 0 } \Rightarrow { x = e }$, for all $e \in E(S)$ and $x \in S_e$. 
In this case, $\| \cdot \|$ is called a \textit{norm}. 
\end{definition}

This definition specializes to groups and sup-semilattices, hence we may talk about a \textit{normed group} or a \textit{normed sup-semilattice}. 
Note that a pseudo-normed sup-semilattice is normed if and only if the map $\| \cdot \|$ is increasing. 

\begin{example}
Let $n$ be a positive integer and $\alpha$ be a real number $\geqslant 1$. 
The usual $L^{\alpha}$-norm $\| \cdot \|_{\alpha}$ on $\mathbb{R}^n$ is defined by 
\[
\| x \|_{\alpha} = (\sum_{i = 1}^n | x_i |^{\alpha})^{1/\alpha}. 
\]
Restricted to the sup-semilattice $\mathbb{R}_+^n$, it is a pseudo-norm in the sense of Definition~\ref{def:pseudonorm}. 
To see why this holds, apply Proposition~\ref{prop:fgsub} to the concave order-preserving function $f : \mathbb{R}_+ \to \mathbb{R}_+, r \mapsto r^{1/\alpha} $ and the finite submodular map $p : \mathbb{R}_+^n \times \mathbb{R}_+^n \to \mathbb{R}_+, (x, y) \mapsto \sum_{i=1}^n (x_i \vee y_i)^{\alpha}$ of Example~\ref{ex:topkins}. 
Moreover, $\| \cdot \|_{\alpha}$ satisfies the axiom of separation, so it is a norm in the sense of Definition~\ref{def:norm}. 

The sup-norm $\| \cdot \|_{\infty}$ defined by 
\[
\| x \|_{\infty} = \max_{1 \leqslant i \leqslant n} | x_i |
\]
is a pseudo-norm on $\mathbb{R}_+^n$ but not a norm in the sense of Definition~\ref{def:norm}, since the axiom of separation fails. 
\end{example}

\begin{example}
Let $E$ be a Riesz space, i.e.\ a lattice-ordered real vector space. 
If $x \in E$, we write $| x |$ for the sup of $\{ x, -x \}$, that is $| x | := x \vee (-x)$. 
Suppose that $E$ admits an order-unit $u$, i.e.\ an element $u \in E$ such that, for every $x \in E$, there exists some $\lambda > 0$ such that $| x | \leqslant \lambda u$. 
Then the map $x \mapsto \| x \|_u := \inf \{ \lambda > 0 : | x | \leqslant \lambda u \}$ is a norm (in the usual sense) on $E$, and all norms of this form are equivalent. 
Moreover, one can easily show that $| x | \leqslant | y | \Rightarrow \| x \|_u \leqslant \| y \|_u$ for all $x, y \in E$, and $\| x \vee y \|_u + \| x \wedge y \|_u \leqslant \| x \|_u + \| y \|_u$ for all $x, y \in E_+$, where $E_+$ denotes the sublattice of $E$ made of nonnegative elements of $E$. 
Thus, $\| \cdot \|_u$ is a pseudo-norm on the lattice $E_+$ in the sense of Definition~\ref{def:pseudonorm}. 
Note however that, in general, this is not a norm on $E_+$ in the sense of Definition~\ref{def:norm} (take for instance $E := \mathbb{R}^2$, partially ordered coordinate-wise, and $u := (1, 1)$; then $\| u \|_u = \| (0, 1) \|_u = 1$, while $(0, 1) < u$). 
\end{example}

A pseudo-norm $\| \cdot \|$ on an inverse semigroup $S$ is called \textit{weakly permutable} if $\| e + x \| = \| x + e \|$, for all $e \in E(S)$, $x \in S$. 
Note that this holds if the pseudo-norm is cyclically permutable, or if $S$ is Clifford (see Lemma~\ref{lem:clifford}). 

\begin{theorem}\label{thm:pseudonorm}
Let $S$ be a pseudo-normed inverse semigroup with a weakly permutable pseudo-norm. 
Let $p$ denote the partial pseudo-metric $(x, y) \mapsto \| x + y^* \|$. 
Then the following conditions are equivalent: 
\begin{enumerate}
  \item\label{thm:pseudonorm1} $\| \cdot \|$ is a norm;
  \item\label{thm:pseudonorm2} $\| \cdot \|_e$ is a norm on $S_e$, for all $e \in E(S)$;  
  \item\label{thm:pseudonorm3} $(S, d_1)$ is Tychonoff;
  \item\label{thm:pseudonorm4} $(S, d_1)$ is Hausdorff; 
  \item\label{thm:pseudonorm5} $(S, d_1)$ is $T_0$;
  \item\label{thm:pseudonorm6} $d_1$ is a metric;
  \item\label{thm:pseudonorm7} $p$ is a partial metric;
  \item\label{thm:pseudonorm8} $\leq_p$ is a partial order; 
  \item\label{thm:pseudonorm9} $\leq_p$ and $\leqslant_S$ agree. 
\end{enumerate}
In this case, $(E(S), \| \cdot \|)$ is a normed sup-semilattice and, if $\| \cdot \|$ is cyclically permutable, then $S$ is a Hausdorff topological inverse semigroup. 
\end{theorem}

\begin{proof}
The equivalence \eqref{thm:pseudonorm1} $\Leftrightarrow$ \eqref{thm:pseudonorm2} is clear from the definitions. 

The equivalences \eqref{thm:pseudonorm3} $\Leftrightarrow$ \eqref{thm:pseudonorm4} $\Leftrightarrow$ \eqref{thm:pseudonorm5} $\Leftrightarrow$ \eqref{thm:pseudonorm6} $\Leftrightarrow$ \eqref{thm:pseudonorm7} $\Leftrightarrow$ \eqref{thm:pseudonorm8} follow from Proposition~\ref{prop:tych}.

\eqref{thm:pseudonorm8} $\Rightarrow$ \eqref{thm:pseudonorm1}. 
Let $e \in E(S)$ and $x \in S_e$, and suppose that $\| x \| = \| e \|$. 
Then $p(x, e) = \| x + e \|  = \| e + x \| = \| x \| = \| e \| = p(e, e)$. 
Also, from $e + x = x$ we get $e \leqslant_S \delta x$, hence $\| e \| \leqslant \| \delta x \| \leqslant \| x \| = \| e \|$, so that $p(x, x) = p(e, e)$. 
This shows that $x \leq_p e$ and $e \leq_p x$, which yields $x = e$. 
So $\| \cdot \|$ is a norm. 

\eqref{thm:pseudonorm1} $\Rightarrow$ \eqref{thm:pseudonorm9}. 
We already know that $x \leqslant_S y \Rightarrow x \leq_p y$. 
For the reverse implication, suppose that $x \leq_p y$, i.e.\ $p(x, y) = p(y, y)$, and let us show that $x \leqslant_S y$. 
Take $e := y + y^*$. 
From $p(x, y) = p(y, y)$ and the weak permutability of $\| \cdot \|$, we get $\| y + x^* + e \| = \| e \|$. 
Since $y + x^* + e \in S_e$ and $\| \cdot \|$ is a norm, we have $y + x^* + e = e$. 
Thus, $y^* + y + x^* + e = y^* + e = y^*$, so that $x^* \leqslant_S y^*$, i.e.\ $x \leqslant_S y$. 

\eqref{thm:pseudonorm9} $\Rightarrow$ \eqref{thm:pseudonorm8} is obvious.  
\end{proof}

\begin{remark}\label{rk:pstar2}
Under the conditions of the previous theorem, we deduce that, if $\leq_p$ is a partial order, then $\leq_{p^*}$ is a partial order that agrees with $\leq_p$. 
Indeed, if $\leq_p$ is a partial order, then $\leq_p$ and $\leqslant_S$ coincide, so that $x \leq_p y \Leftrightarrow x \leqslant_S y \Leftrightarrow x^* \leqslant_S y^* \Leftrightarrow x^* \leq_p y^* \Leftrightarrow x \leq_{p^*} y$. 
\end{remark}

\begin{remark}
The equivalence \eqref{thm:pseudonorm8} $\Leftrightarrow$ \eqref{thm:pseudonorm9} also holds if one replaces the weak permutability of $\| \cdot \|$ by the property 
\[
\| x + x^* \| = \| x^* + x \|,
\]
for all $x \in S$. 
(Note that, just like the property of weak permutability, this property holds if the pseudo-norm is cyclically permutable, or if $S$ is Clifford.) 
Indeed, suppose that $x \leq_p y$, and let us show that $x \leqslant_S y$, i.e.\ that $y = z$ where we take $z := x + y^* + y$. 
We have $p(y, y) = \| y + y^* \| = \| y^* + y \| \leqslant \| z^* + z \| = \| z + z^* \| = \| x + y^* + y + x^* \| = \| \delta(x + y^*) \| \leqslant \| x + y^* \| = p(x, y) = p(y, y)$. 
Thus, $p(y, y) = \| z + z^* \| = p(z, z)$. 
Moreover, $z + y^* = x + y^*$, so that $p(z, y) = p(x, y) = p(y, y) = p(z, z)$. 
This shows that $y \leq_p z$ and $z \leq_p y$, so $y = z$ since $\leq_p$ is a partial order. 
\end{remark}

\begin{remark}
If $S$ is a pseudo-normed inverse semigroup with a weakly permutable pseudo-norm, then the map $\delta$ is continuous: the condition of cyclic permutability used in Corollary~\ref{coro:topo} is actually not necessary. 
Indeed, let $x \in S$ and $(x_n)_n$ be a net with $x_n \to x$. 
We have $\| x_n + x^* \| \to \| \delta x \|$ and $\| \delta x_n \| \to \| \delta x \|$. 
We need to show that $\delta x_n \to \delta x$, so it remains to prove that $\| \delta x_n + \delta x \| \to \| \delta x \|$. 
Now, using submodularity, 
\[
 \| \delta x_n + \delta x \| + \| \delta x \| \leqslant \| x_n + x^* \| + \| x + x_n^* + \delta x \|. 
\]
By weak permutability, we have $\| x + x_n^* + \delta x \| = \| x + x_n^* \| = \| x_n + x^* \|$. 
Thus, 
\[
\| \delta x \| \leqslant \| \delta x_n + \delta x \| \leqslant 2 \| x_n + x^* \| - \| \delta x \|,
\]
so that $\| \delta x_n + \delta x \| \to \| \delta x \|$, as required. 
\end{remark}

\begin{definition}
A \textit{Banach inverse semigroup} (resp.\ a \textit{Banach group}, a \textit{Banach sup-semilattice}) is a normed inverse semigroup (resp.\ a normed group, a normed sup-semilattice) with a cyclically permutable norm, Cauchy complete with respect to its intrinsic metric $d_1$. 
\end{definition}




In a Banach inverse semigroup, the notion of monotone net is unambiguous, since the partial orders $\leq_p$, $\leq_{p^*}$, and $\leqslant_S$ coincide. 

\begin{theorem}\label{thm:cv}
Let $S$ be a Banach inverse semigroup. 
Then $S$ equipped with the norm topology is a Hausdorff topological inverse semigroup. 
In addition, $S$ is conditionally monotone-complete, and the partial order is compatible with the norm topology. 
Moreover, if $(x_n)_n$ is a monotone net in $S$, then the following conditions are equivalent:
\begin{enumerate}
  \item\label{thm:cv1} $(x_n)_n$ converges;
  \item\label{thm:cv1bis0} $(x_n)_n$ has a supinf;
  \item\label{thm:cv1bis} $\{ x_n \}_n$ is bounded;
  \item\label{thm:cv2} $\{ \| x_n \| \}_n$ is bounded;
  \item\label{thm:cv3} $\{ \| \delta x_n \| \}_n$ is bounded. 
\end{enumerate}
\end{theorem}

\begin{proof}
The fact that $S$ is conditionally monotone-complete and the equivalences \eqref{thm:cv1} $\Leftrightarrow$ \eqref{thm:cv1bis0} $\Leftrightarrow$ \eqref{thm:cv1bis} $\Leftrightarrow$ \eqref{thm:cv3} are consequences of Theorem~\ref{thm:cvi}. 
The implication \eqref{thm:cv1bis} $\Rightarrow$ \eqref{thm:cv2} follows from Proposition~\ref{prop:normproperties}. 
The implication \eqref{thm:cv2} $\Rightarrow$ \eqref{thm:cv3} is clear from the inequality $\| \delta x \| \leqslant \| x \|$. 
\end{proof}

\begin{remark}
By the previous theorem, every Banach inverse semigroup is conditionally monotone-complete, hence \textit{mirror} in the sense of Poncet \cite[Definition~3.1]{Poncet12a}. 
\end{remark}

\section{The special case of Clifford monoids}\label{sec:clifford}

In this last section, we examine the correspondence between pseudo-norms and right-subinvariant pseudo-metrics for Clifford monoids. 

\begin{definition}\label{def:skew}
Let $S$ be a Clifford semigroup and $d$ be a pseudo-metric on $S$. 
We say that $d$ is \textit{skew-convex} if the following conditions hold:
\begin{enumerate}
  \item $d(x, y) \geqslant d(\delta x, \delta y)$, for all $x, y \in S$, and
  \item $d(x, z) = d(\delta x, \delta y) + d(x + y^*, z + y^*)$, for all $x, y, z \in S$ with $\delta x \leqslant_S \delta y \leqslant_S \delta z$. 
\end{enumerate}
\end{definition}

\begin{remark}
If $S$ reduces to a group, then the partial order $\leqslant_S$ agrees with equality, so radial-convexity is trivial, and right-subinvariance and skew-convexity coincide. 
If $S$ reduces to a sup-semilattice, then radial-convexity and skew-convexity coincide. 
\end{remark}

\begin{lemma}
Let $S$ be a Clifford semigroup equipped with a pseudo-metric $d$. 
If $d$ is skew-convex, then $d$ is radially-convex. 
\end{lemma}

\begin{proof}
Let $x, y, z \in S$ with $x \leqslant_S y \leqslant_S z$. 
Then $\delta x \leqslant_S \delta y \leqslant_S \delta z$. 
Using skew-convexity, we have $d(x, z) = d(\delta x, \delta y) + d(x + y^*, z + y^*)$.
Using again skew-convexity, we also have 
\[
d(x, y) + d(y, z) = d(\delta x, \delta y) + d(x + y^*, y + y^*) + d(y + y^*, z + y^*). 
\]
Now, $x \leqslant_S y$ gives $x + y^* = y + y^*$. 
This yields $d(x, z) = d(x, y) + d(y, z)$, as required. 
%
\end{proof}

\begin{proposition}\label{prop:dclifford}
Let $S$ be a pseudo-normed Clifford semigroup. 
Then the pseudo-metrics $d_0$ and $d_1$ are right-subinvariant and skew-convex. 
\end{proposition}

\begin{proof}
We already know from Proposition~\ref{prop:propd}\eqref{prop:propd2} that $d_0$ and $d_1$ are right-subinvariant. 
Let us show that $d_1$ is skew-convex (the proof for $d_0$ is similar). 
So let $x, y \in S$. 
We have $\| x + y^* \| \geqslant \| \delta(x + y^*) \| = \| \delta x + \delta y \|$. 
This yields $d_1(x, y) \geqslant d_1(\delta x, \delta y)$, so the first axiom of skew-convexity is fulfilled. 

Now, let $x, y, z \in S$ with $\delta x \leqslant_S \delta y \leqslant_S \delta z$. 
Then $\delta y^* + z^* = z^*$, and 
\begin{align*}
d_1(\delta x, \delta y)  &= \frac{\| \delta y \| - \| \delta x \| }{2} \\
d_1(x + y^*, z + y^*) &= \| x + \delta y^* + z^* \| - \frac{ \| \delta(x + y^*) \|  + \| \delta(z + y^*) \| }{2} \\
&= \| x + z^* \| - \frac{ \| \delta y \|  + \| \delta z \| }{2},
\end{align*}
so that 
\begin{align*}
d_1(x, z) &= \| x + z^* \| - \frac{ \| \delta x \| + \| \delta z \| }{2} \\
&= d_1(\delta x, \delta y) + d_1(x + y^*, z + y^*),
\end{align*}
as required. 
\end{proof}

The following result gives a converse statement to Proposition~\ref{prop:dclifford}. 

\begin{theorem}\label{thm:exist}
Let $(S, +, 0)$ be a Clifford monoid equipped with a right-subinvariant, skew-convex pseudo-metric $d$. 
Then the map $v : S \to \mathbb{R}_+$ defined by
\[
v(x) = { d(x, 0) + d(\delta x, 0) },
\]
for all $x \in S$, is a pseudo-norm on $S$ such that the metric topology and the norm topology coincide. 
Moreover, $v$ is a norm if and only if $d$ is a metric. 
\end{theorem}

\begin{proof}
We first prove that $v(x) = v(x^*)$, for all $x \in S$. 
Let $x \in S$. 
Using skew-convexity and the fact that $0 \leqslant_S \delta x$, we have 
$d(0, x) = d(0, \delta x) + d(0 + \delta x, x + \delta x)$, i.e.\ $v(x) = v(\delta x) + d(\delta x, x)$. 
Recalling that $S$ is Clifford, we also have $d(0, x^*) = d(0, \delta x^*) + d(0 + x, x^* + x)$, i.e.\ $v(x^*) = v(\delta x^*) + d(x, \delta x^*) = v(\delta x) + d(x, \delta x)$. 
This proves that $v(x) = v(x^*)$, as required. 

We now show that $v$ is order-preserving. 
Let $x, y \in S$ with $x \leqslant_S y$. 
Using the triangle inequality, we have $d(0, x) \leqslant d(0, \delta y^*) + d(\delta y^*, x)$, hence $v(x) \leqslant d(0, \delta x) + d(\delta y^*, x) + d(0, \delta y^*)$. 
Moreover, by skew-convexity, 
$d(0, y^*) = d(0, \delta x) + d(0 + x, y^* + x)$. 
Since $x \leqslant_S y$, we have $y^* + x = \delta y^*$. 
Thus, $v(y^*) = d(0, y^*) + d(0, \delta y^*) = d(0, \delta x) + d(x, \delta y^*) + d(0, \delta y^*)$. 
This yields $v(x) \leqslant v(y^*) = v(y)$, as required. 

Now let us consider the map $d'$ defined by $d'(x, y) = d(x, y) + d(\delta x, \delta y)$, for all $x, y \in S$. 
It is straightforward to show that $d'$ is a right-subinvariant, skew-convex pseudo-metric on $S$. 
Moreover, $d \leqslant d' \leqslant 2 d$, so $d$ and $d'$ are equivalent: they generate the same topology, and $d$ is a metric if and only if $d'$ is a metric. 
All in all, we can now reason on $d'$ only. 
Note also that $v(x) = d'(x, 0)$, for all $x \in S$. 

It is easily seen that $v$ is subadditive thanks to right-subinvariance: 
\begin{align*}
v(x + y) = d'(x + y, 0) &\leqslant d'(x + y, y) + d'(y, 0) \\
&\leqslant d'(x, 0) + d'(y, 0) = v(x) + v(y),
\end{align*}
for all $x, y \in S$. 
However, this is not sufficient to prove that $v$ is a pseudo-norm. 
We shall prove one by one the three axioms listed in Definition~\ref{def:pseudonorm}. 

(Axiom~$1$). 
We already know that $v(x) = v(x^*)$, for all $x \in S$. 
Thus, $v(x + y^*) = v(y + x^*)$, for all $x, y \in S$. 

(Axiom~$2$). 
Let $x, y \in S$, and take $z := x + y^*$. 
Then $v(z) \geqslant v(\delta z)$. 
Moreover, $\delta x \leqslant_S \delta z$. 
Since $v$ is order-preserving, this yields $v(z) \geqslant v(\delta x)$. 
Thus, $v(\delta x) \leqslant v(x + y^*)$. 

(Axiom~$3$). 
Let $x, y, z \in S$, and take $t := x + z^* + z$.  
We have 
\[
v(t + y^*) = d'(t + y^*, 0) \leqslant d'(t + y^*, z + y^*) + d'(z + y^*, 0). 
\]
With the right-subinvariance of $d'$, we get $v(t + y^*) \leqslant d'(t, z) + v(z + y^*)$. 
Thus, $v(t + y^*) + v(z + z^*) \leqslant d'(t, z) + d'(\delta z, 0) + v(z + y^*)$. 
Note that $0 \leqslant_S \delta z \leqslant_S \delta t$, since $S$ is Clifford; using skew-convexity, we deduce that $d'(t, z) + d'(\delta z, 0) = d'(t + z^*, 0) = v(t + z^*) = v(x + z^*)$. 
This gives 
\[
v(t + y^*) + v(z + z^*) \leqslant v(x + z^*) + v(z + y^*). 
\]
Now $v$ is order-preserving and $x + y^* \leqslant t + y^*$, so $v(x + y^*) \leqslant v(t + y^*)$. 
So $v$ is submodular. 

We have proved the three axioms of Definition~\ref{def:pseudonorm}; moreover, $v(0) = 0$, so $v$ is indeed a pseudo-norm on the inverse monoid $S$. 

Now we show that the metric topology generated by $d'$ and the norm topology generated by $v$ coincide. 
Note that $v$ is continuous for the metric topology since, by the triangle inequality, we have 
\[
| v(x) - v(y) | \leqslant d'(x, y),
\]
for all $x, y \in S$. 
Let $(x_n)_n$ be a net that tends to some $x \in S$ for the metric topology, i.e.\ $d'(x_n, x) \to 0$. 
Since $d'$ is right-subinvariant, $d'(x_n, x) \geqslant d'(x_n + x^*, \delta x)$, so $x_n + x^* \to \delta x$. 
By continuity of $v$, we deduce that $v(x_n + x^*) \to v(\delta x)$. 
Moreover, 
\[
v(x + x_n^*) - d'(x_n, x) \leqslant v(\delta x_n) \leqslant d'(x_n, x) + v(x + x_n^*), 
\]
so that $v(\delta x_n) \to v(\delta x)$. 
By Theorem~\ref{thm:dist1}, this proves that $(x_n)_n$ tends to $x$ with respect to the norm topology generated by $v$. 

Conversely, let $(x_n)_n$ be a net that tends to some $x \in S$ for the norm topology, i.e.\ $v(x_n + x^*) \to v(\delta x)$ and $v(\delta x_n) \to v(\delta x)$. 
By the triangle inequality, we have $d'(x_n, x) \leqslant A_n + B_n$, where $A_n := d'(x_n, x + \delta x_n)$ and $B_n := d'(x + \delta x_n, x)$. 
Using right-subinvariance, 
\begin{align*}
A_n &= d'(\delta x_n + x_n, x + x_n^* + x_n) \\
&\leqslant d'(\delta x_n, x + x_n^*) \\
&\leqslant d'(\delta x_n, \delta x) + d'(\delta x, x + x_n^*),
\end{align*}
and 
\begin{align*}
B_n &= d'(x + \delta x_n + x^* + x, \delta x + x) \\
&\leqslant d'(x + \delta x_n + x^*, \delta x) \\
&= d'(\delta x_n + \delta x, \delta x + \delta x) \\
&\leqslant d'(\delta x_n, \delta x). 
\end{align*}
It remains to show that $d'(\delta x_n, \delta x) \to 0$ and $d'(\delta x, x + x_n^*) \to 0$. 
By skew-convexity, $v(x_n + x^*) = v(\delta x) + d'(\delta x, x_n + x^*)$, hence $d'(\delta x, x + x_n^*) \to 0$. 
Also, by radial-convexity,
\begin{align*}
d'(\delta x_n, \delta x) &\leqslant d'(\delta x_n, \delta x_n + \delta x) + d'(\delta x_n + \delta x, \delta x) \\
&= 2 v(\delta x_n + \delta x) - v(\delta x_n) - v(\delta x) \\
&\leqslant 2 v(x_n + x^*) - v(\delta x_n) - v(\delta x),
\end{align*}
and this latter term tends to $0$. 

Suppose that $d'$ is a metric, and let us show that $v$ is a norm. 
So suppose that $v(x) = v(e)$, for some $x \in S_e$ and $e \in E(S)$. 
Then $d'(0, x) = d'(0, e)$. 
Using skew-convexity, this amounts to $d'(0, \delta x) + d'(\delta x, x) = d'(0, e)$. 
Since $0 \leqslant_S e \leqslant_S \delta x$ and $d'$ is radially-convex, we have $d'(0, \delta x) = d'(0, e) + d'(e, \delta x)$, so that $d'(e, \delta x) + d'(\delta x, x) = 0$. 
Since $d'$ is a metric, this gives $e = \delta x = x$, as required. 

Now suppose that $v$ is a norm, and let us show that $d'$ is a metric. 
So suppose that $d'(x, y) = 0$, for some $x, y \in S$. 
Take $z := x + y^*$. 
By skew-convexity, $d'(0, z) = d'(0, \delta y) + d'(0 + \delta y, z + \delta y)$, so $d'(\delta y, z) = v(z) - v(\delta y)$. 
By right-subinvariance, $d'(\delta y, z) \leqslant d'(y, x) = 0$, hence $v(z) = v(\delta y)$. 
Since $v$ is a norm and $z \in S_{\delta y}$, we obtain $z = \delta y$, i.e.\ $x + y^* = \delta y$. 
This proves that $x \leqslant_S y$. 
Similarly, we have $y \leqslant_S x$, so $x = y$. 
Thus, $d'$ is a metric. 
\end{proof}

\begin{corollary}
Let $(G, +, 0)$ be a group equipped with a right-invariant pseudo-metric $d$. 
Then the map $v : G \to \mathbb{R}_+$ defined by
$
v(x) = { d(x, 0) }
$, 
for all $x \in G$, is a pseudo-norm on $G$ such that the metric topology and the norm topology coincide. 
Moreover, $v$ is a norm if and only if $d$ is a metric. 
\end{corollary}

\begin{corollary}
Let $(S, \vee, 0)$ be a sup-semilattice with $0$ equipped with a right-subinvariant, radially-convex pseudo-metric $d$. 
Then the map $v : S \to \mathbb{R}_+$ defined by
$
v(x) = { d(x, 0) }
$, 
for all $x \in S$, is a pseudo-norm on $S$ such that the metric topology and the norm topology coincide. 
Moreover, $v$ is a norm if and only if $d$ is a metric. 
\end{corollary}

\section{Conclusion and perspectives}

By the Birkhoff--Kakutani theorem, every first-countable Hausdorff topological group has a right-invariant metric compatible with the topology, see Birkhoff \cite{Birkhoff36}, Kakutani \cite{Kakutani36}. 
As a next step, we shall investigate possible generalizations of this result to topological inverse semigroups. 

We also think of generalizing further the concepts of this paper and define norms on regular involutive semigroups. 
In this process, we would certainly need to replace idempotent elements by \textit{projections}, i.e.\ elements $e$ of the regular involutive semigroup such that $e + e = e$ and $e^* = e$. 

Another perspective is to apply normed inverse semigroups and normed sup-semilattices to idempotent analysis, as explained in the Introduction. 

\bibliographystyle{plain}

\def\cprime{$'$} \def\cprime{$'$} \def\cprime{$'$} \def\cprime{$'$}
  \def\ocirc#1{\ifmmode\setbox0=\hbox{$#1$}\dimen0=\ht0 \advance\dimen0
  by1pt\rlap{\hbox to\wd0{\hss\raise\dimen0
  \hbox{\hskip.2em$\scriptscriptstyle\circ$}\hss}}#1\else {\accent"17 #1}\fi}
  \def\ocirc#1{\ifmmode\setbox0=\hbox{$#1$}\dimen0=\ht0 \advance\dimen0
  by1pt\rlap{\hbox to\wd0{\hss\raise\dimen0
  \hbox{\hskip.2em$\scriptscriptstyle\circ$}\hss}}#1\else {\accent"17 #1}\fi}

\end{document}